\documentclass{article}
\usepackage{url}
\usepackage{todonotes}
\usepackage{amsthm, amsmath, amssymb, float, doi, subcaption}
\usepackage[capitalise]{cleveref}
\usepackage{pgfplots}

\usepackage{adjustbox}

    \usepackage{todonotes}
    \usepackage[]{geometry}
    
        \definecolor{johns}{HTML}{CBC3E3}
        \definecolor{beths}{HTML}{B7FFFA}
        \definecolor{seans}{HTML}{98FB98}

\newtheorem{theorem}{Theorem}[section]

\newtheorem{lemma}[theorem]{Lemma}
\newtheorem{observation}[theorem]{Observation}

\newtheorem{conjecture}[theorem]{Conjecture}
\newtheorem{proposition}[theorem]{Proposition}

\newtheorem{construction}[]{Construction}
\newtheorem{claim}[theorem]{Claim}
\newtheorem{question}[theorem]{Question}

\newtheorem{innercustomgeneric}{\customgenericname}
\providecommand{\customgenericname}{}
\newcommand{\newcustomtheorem}[2]{%
  \newenvironment{#1}[1]
  {%
   \ifdefined\crefalias\crefalias{innercustomgeneric}{#2}\fi
   \renewcommand\customgenericname{#2}%
   \renewcommand\theinnercustomgeneric{##1}%
   \innercustomgeneric
  }
  {\endinnercustomgeneric}%
  \ifdefined\crefname\crefname{#2}{#2}{#2s}\fi
}

\newcustomtheorem{customthm}{Theorem}

\newcommand{\glue}[3]{{#1}\boxminus_{{#2}}{#3}}

\DeclareMathOperator*{\argmax}{arg\,max}

\tikzstyle{Black Node}=[
    draw=black,
    fill=white,
    text=black,
    shape=circle,
    minimum size=2em
]
\tikzstyle{Small Black Node}=[
    draw=black,
    fill=white,
    text=black,
    shape=circle,
    minimum size=1em
]
\tikzstyle{Blue Node}=[
    draw=black,
    fill=blue!40,
    text=blue!80!black,
    shape=circle,
    minimum size=2em
]
\tikzstyle{Small Blue Node}=[
    draw=black,
    fill=blue!40,
    text=blue!80!black,
    shape=circle,
    minimum size=1em
]
\tikzstyle{Red Node}=[
    draw=black,
    fill=red!30,
    text=red!50!black,
    shape=circle,
    minimum size=2em
]
\tikzstyle{Small Red Node}=[
    draw=black,
    fill=red!30,
    text=red!50!black,
    shape=circle,
    minimum size=1em
]
\tikzstyle{Green Node}=[
    draw=black,
    fill=green!30,
    text=green!50!black,
    shape=circle,
    minimum size=2em
]
\tikzstyle{Small Green Node}=[
    draw=black,
    fill=green!30,
    text=green!50!black,
    shape=circle,
    minimum size=1em
]
\tikzstyle{Text Node}=[
    draw=none,
    fill=none,
    anchor=west,
    shape=rectangle,
    inner sep=0pt,
    outer sep=0pt,
    text height=9pt
]
\tikzstyle{Black Edge}=[draw=black, -]
\tikzstyle{Black Alpha Edge}=[draw=black, opacity=0.5, -]
\tikzstyle{Dashed Edge}=[draw=black, dashed, opacity=0.5]

\title{Cost-Benefit Analysis for PMU Placement in Power Grids}

\author{Beth Bjorkman\footnote{Air Force Research Laboratory} \and Sean English\footnote{University of North Carolina Wilmington, \texttt{EnglishS@uncw.edu}} \and Johnathan Koch\footnote{Applied Research Solutions}}

\date{27 Jan 2026}
\pgfplotsset{compat=1.18}
\begin{document}

\maketitle

\begin{abstract}
    Power domination is a graph-theoretic model for the observance of a power grid using phasor measurement units (PMUs).
    There are many costs associated with the installation of a PMU, but also costs associated with not observing the entire power grid.
    In this work, we propose and study a power domination cost function, which balances these two costs.
    Given a graph $G$, a set of sensor locations $S$, and a parameter $\beta$ (which is the ratio of the cost of a PMU to the cost of non-observance of any given vertex), we define the cost function
    \[ \mathrm{C}(G;S,\beta)=|S|+\beta\cdot (|V(G)|-|\mathrm{Obs}(G;S)|) \]
    where $|\mathrm{Obs}(G;S)|$ is the number of vertices observed by sensors placed at $S\subseteq V(G)$ in the power domination process.
    
    We explore the values of $k$ for which there is a set $S$ of size $k$ that minimizes this cost function, and explore which values of $\beta$ guarantee that it is optimal to observe the entire power grid to minimize cost.
    We also introduce notions of marginal cost and marginal observance, providing tools to analyze how many PMUs one should install on a given power grid.
\end{abstract}

\section{Introduction}
    The power domination process is a graph theoretic propagation process designed to model observation of a power grid using \emph{phasor measurement units} (PMUs)~\cite{HHHH2002}.
    A PMU placed at a vertex in a graph $G$ will \textbf{observe} the vertex it is placed at and all vertices connected to it.
    After this initial round of observance, Kirchoff's laws can be used to observe more of the graph; if an observed vertex has exactly one unobserved neighbor, then this neighbor becomes observed. 
    
    The most well-studied question in power domination is that of determining the \textbf{power domination number}.
    Given a graph $G$, the power domination number, $\gamma_P(G)$, is the least number of sensors necessary to observe all vertices of $G$.
    In reality, installing PMUs can be cost-prohibitive.
    For example, in 2021 the Independent Electricity System Operator (IESO) in Ontario, Canada estimated that the installation cost of a PMU is between \$50,000 and \$300,000 CAD, depending on many factors\cite{IESO2021a}.
    However, the high cost of a PMU is offset by the information they provide.
    For example, the IESO cites benefits of PMU observance such as reliability and power grid resiliency, efficiency and cost-savings, and environmental benefits\cite{IESO2021b}.

    Other works have proposed models that optimize for redundancy \cite{PCW2010, BCF2023} and information propagation time \cite{FHK2016, BCFHS2019}.
    We propose a model that weighs the cost of a PMU against the benefits of higher graph observance, to determine PMU placements which may not observe the entire graph, but instead optimize the costs and benefits.
    
    \subsection{The power domination process and cost function}
        Formally, given a graph $G$ and a set of PMU sensor placements $S\subseteq V(G)$, the power domination process is defined recursively in two steps as follows:
        \begin{enumerate}
        	\item
                \emph{(Domination Step)} Given the initial placement $S$, let $B:=N[S]$.
        	\item
                \emph{(Zero Forcing Step)} If there exists a vertex $x\in B$ with a single neighbor $y\in V(G)\setminus B$, then add $y$ to $B$.
                In this manner, the vertex $x$ forces $y$.
                Repeat this step until no such vertex $x\in B$ with a single neighbor $y\in V(G)\setminus B$ exists.
        \end{enumerate}
        
        We call the final set $B$ after performing the power domination process, starting from initial sensor placement $S$, the \emph{observed set} and denote it with $\mathrm{Obs}(G;S)$.
        If $\mathrm{Obs}(G;S)=V(G)$, then $S$ a \emph{power dominating set}.
        The size of a minimum power dominating set is the \emph{power domination number} and is denoted by $\gamma_P(G)$.
        
        We make two simplifying assumptions; 1. We assume that the cost of placing at PMU at any vertex in $G$ is the same, and 2. We assume that the cost associated with not observing a vertex is the same for every vertex in $V(G)$.
        We further normalize our total cost to be in terms of the price of a single PMU.
        With this in mind, we define the \emph{Observance Cost Ratio} $\beta$, which is the ratio of the cost of not observing a vertex to the cost of a PMU.
        This leads us to our \emph{cost function} for the graph $G$ with initial sensor placement $S$ and Observance Cost Ratio $\beta$:
        \[
            \mathrm{C}(G;S,\beta)=|S|+\beta\cdot (|V(G)|-|\mathrm{Obs}(G;S)|).
        \]
        
        We do not attempt to estimate $\beta$ in this work.
        Instead, we treat $\beta$ as an unknown parameter, and attempt to minimize the cost function over all sensor placements in a given graph $G$, over the entire range of possible choices for $\beta\in \mathbb{R}_{\geq 0}$.
        In this way, $\mathrm{C}(G;S,\beta)$ is a linear function in $\beta$.
        For a fixed $\beta$, we will call a set $S\subseteq V(G)$ \emph{$\beta$-best} if 
        	\[
        	\mathrm{C}(G;S,\beta)=\min_{S'\subseteq  V(G)} \mathrm{C}(G;S',\beta).
        	\]
        That is, if we know the true observance cost ratio is $\beta$, then a $\beta$-best set is an optimal sensor placement that minimizes the total cost of PMUs and not observing portions of the graph.
        It is worth noting that for a given $\beta$, there may be more than one $\beta$-best set.
        In general, $\beta$-best sets rely heavily on the structure of the graph $G$, but $\beta$-best sets for extreme values of $\beta$ are easily characterized.
        
        \begin{observation}\label{observation beta small or large}
        	If $\beta\leq \frac{1}{|V(G)|}$, then the empty set (i.e. using no sensors) is $\beta$-best, while on the other extreme, if $\beta\geq 1$, then any minimum power dominating set is $\beta$-best.
        \end{observation}
        
        If we have two sets $S,S'\subseteq V(G)$, with $|S|=|S'|$, then whichever of the two sets observes more will have a lower total cost under our model.
        As such, for a fixed $k\in \mathbb{N}$, we define the \emph{maximum observance} for the graph $G$ and integer $k$ as 
        \[
            \mathrm{maxObs}(G;k):=\max_{S\in \binom{V(G)}{k}}|\mathrm{Obs}(G;S)|.
        \]
        The power domination number $\gamma_P(G)$ is the smallest $k$ such that $\mathrm{maxObs}(G;k)=|V(G)|$. It is worth noting that $\mathrm{maxObs}$ is strictly increasing for $0\leq k\leq \gamma_P(G)$, as formalized below.
        
        \begin{observation}\label{observation more sensors means more observed}
        	For a graph $G$ and $i,j\in [\gamma_P(G)]\cup \{0\}$ with $i<j$, it holds that $\mathrm{maxObs}(G;i)<\mathrm{maxObs}(G;j)$.
        \end{observation}
        
        When optimizing cost over all sensor placements of a fixed size $k$ on the graph $G$, we will sometimes write
        \[
        \mathrm{C}(G;k,\beta):=k+\beta(|V(G)|-\mathrm{maxObs}(G;k)).
        \]
        
        It is straightforward to see that if a set $S$ is $\beta$-best for some $\beta$, then $|\mathrm{Obs}(G;S)|=\mathrm{maxObs}(G;|S|)$.
        However, a set being $\beta$-best for a particular $\beta$ is strictly stronger than it maximizing the observance for a fixed size: $\mathrm{C}(G;S,\beta) \leq \mathrm{C}(G;S',\beta)$ for any $S'\subseteq V(G)$.
        We introduce the \emph{minimum cost function} for $G$:
        \[
        \mathrm{minC}(G;\beta):=\min_{S\subseteq V(G)}\mathrm{C}(G;S,\beta).
        \]
        Similar to the cost function, we will usually think of $\mathrm{minC}$ as a function of $\beta$, considering the graph $G$ to be fixed.
        By Observation~\ref{observation beta small or large}, when $\beta=0$, we have that $\mathrm{minC}(G;0)=0$, and if $\beta\geq 1$, then $\mathrm{minC}(G;1)=\gamma_P(G)$ and so we will focus on the range $\beta\in (0,1)$.
        In this range, $\mathrm{minC}$ will be piecewise linear, representing $\beta$-best sets for sub-intervals within $(0,1)$.
        
        It may be the case that a set is $\beta$-best only for a single value, say $\beta^*$.
        In these cases, there is always another set that is $\beta$-best on a non-degenerate interval $I\subseteq (0,1)$ with $\beta^*\in I$.
        As such, we will call a set $S\subseteq V(G)$ \emph{useful} for $G$ if there exists a non-degenerate interval $I$ such that the set $S$ is $\beta$-best for all $\beta\in I$.
        We will call some size $k\in [\gamma_P(G)]\cup \{0\}$ \emph{useful} if there exists a useful set of size $k$.

    \subsection{Example of useful sizes and \texorpdfstring{$\beta$}{beta}-best sets}
        We now present an example cost-benefit analysis on a graph used in optimal power flow simulations.
        The graph pictured in Figure~\ref{fig:60nodesystem} stems from the Nordic32 test system initially developed by CIGRE Task Force in 1995 \cite{S1995} and represents a potential power grid for a Nordic country \cite{C2020}.
        It contains $60$ nodes and has power domination number $11$.
        In Table~\ref{tab:betabestsets}, the maximum observance and cost functions for this graph are calculated for each size $k$ from $0$ to $11$, along with a representative set that realizes the maximum observance.
        The sizes $k=0,1,4,6,9$ and $11$ are all useful, whereas $2,3,5,7,8$ and $10$ are not.
        Figure~\ref{figure plot of cost functions for 60 node system} plots the $12$ different minimal cost functions for each size, showing how the minimum cost sets evolve as $\beta$ varies.

        \begin{figure}
        	\centering
            \resizebox{.5\textwidth}{!}{
                \begin{tikzpicture}
                    \node [style=Black Node] (0) at (6, 9) {$0$};
                    \node [style=Black Node] (1) at (4, 9) {$1$};
                    \node [style=Black Node] (2) at (5, 9) {$2$};
                    \node [style=Black Node] (3) at (4, 11) {$3$};
                    \node [style=Black Node] (4) at (1, 11) {$4$};
                    \node [style=Black Node] (5) at (1, 2) {$5$};
                    \node [style=Black Node] (6) at (1, 1) {$6$};
                    \node [style=Black Node] (7) at (4, 7) {$7$};
                    \node [style=Black Node] (8) at (5, 6) {$8$};
                    \node [style=Black Node] (9) at (9, 8) {$9$};
                    \node [style=Black Node] (10) at (7, 5) {$10$};
                    \node [style=Black Node] (11) at (0, 11) {$11$};
                    \node [style=Black Node] (12) at (2, 7) {$12$};
                    \node [style=Black Node] (13) at (7, 3) {$13$};
                    \node [style=Black Node] (14) at (7, 0) {$14$};
                    \node [style=Black Node] (15) at (6, 0) {$15$};
                    \node [style=Black Node] (16) at (4, 1) {$16$};
                    \node [style=Black Node] (17) at (3, 2) {$17$};
                    \node [style=Black Node] (18) at (1, 6) {$18$};
                    \node [style=Black Node] (19) at (1, 4) {$19$};
                    \node [style=Black Node] (20) at (0, 4) {$20$};
                    \node [style=Black Node] (21) at (9, 5) {$21$};
                    \node [style=Black Node] (22) at (8, 7) {$22$};
                    \node [style=Black Node] (23) at (7, 9) {$23$};
                    \node [style=Black Node] (24) at (5, 8) {$24$};
                    \node [style=Black Node] (25) at (2, 1) {$25$};
                    \node [style=Black Node] (26) at (2, 9) {$26$};
                    \node [style=Black Node] (27) at (3, 8) {$27$};
                    \node [style=Black Node] (28) at (9, 7) {$28$};
                    \node [style=Blue Node] (29) at (6, 3) {$29$};
                    \node [style=Blue Node] (30) at (6, 1) {$30$};
                    \node [style=Black Node] (31) at (0, 7) {$31$};
                    \node [style=Black Node] (32) at (0, 6) {$32$};
                    \node [style=Blue Node] (33) at (0, 5) {$33$};
                    \node [style=Blue Node] (34) at (6, 8) {$34$};
                    \node [style=Black Node] (35) at (3, 11) {$35$};
                    \node [style=Blue Node] (36) at (6, 6) {$36$};
                    \node [style=Black Node] (37) at (6, 5) {$37$};
                    \node [style=Black Node] (38) at (2, 10) {$38$};
                    \node [style=Black Node] (39) at (5, 5) {$39$};
                    \node [style=Black Node] (40) at (7, 4) {$40$};
                    \node [style=Blue Node] (41) at (0, 10) {$41$};
                    \node [style=Black Node] (42) at (0, 8) {$42$};
                    \node [style=Black Node] (43) at (1, 8) {$43$};
                    \node [style=Black Node] (44) at (6, 4) {$44$};
                    \node [style=Black Node] (45) at (5, 4) {$45$};
                    \node [style=Black Node] (46) at (5, 3) {$46$};
                    \node [style=Black Node] (47) at (3, 3) {$47$};
                    \node [style=Black Node] (48) at (6, 7) {$48$};
                    \node [style=Blue Node] (49) at (2, 11) {$49$};
                    \node [style=Black Node] (50) at (0, 3) {$50$};
                    \node [style=Blue Node] (51) at (0, 0) {$51$};
                    \node [style=Black Node] (52) at (5, 0) {$52$};
                    \node [style=Black Node] (53) at (2, 3) {$53$};
                    \node [style=Blue Node] (54) at (10, 6) {$54$};
                    \node [style=Blue Node] (55) at (1, 7) {$55$};
                    \node [style=Black Node] (56) at (6, 2) {$56$};
                    \node [style=Black Node] (57) at (7, 1) {$57$};
                    \node [style=Blue Node] (58) at (4, 2) {$58$};
                    \node [style=Black Node] (59) at (10, 10) {$59$};
                    \draw [style=Black Edge] (0) to (23);
                    \draw [style=Black Edge] (1) to (24);
                    \draw [style=Black Edge] (2) to (34);
                    \draw [style=Black Edge] (3) to (35);
                    \draw [style=Black Edge] (4) to (49);
                    \draw [style=Black Edge] (5) to (25);
                    \draw [style=Black Edge] (6) to (51);
                    \draw [style=Black Edge] (7) to (27);
                    \draw [style=Black Edge] (8) to (36);
                    \draw [style=Black Edge] (9) to (59);
                    \draw [style=Black Edge] (10) to (37);
                    \draw [style=Black Edge] (11) to (41);
                    \draw [style=Black Edge] (12) to (55);
                    \draw [style=Black Edge] (13) to (29);
                    \draw [style=Black Edge] (14) to (30);
                    \draw [style=Black Edge] (15) to (30);
                    \draw [style=Black Edge] (16) to (58);
                    \draw [style=Black Edge] (17) to (58);
                    \draw [style=Black Edge] (18) to (32);
                    \draw [style=Black Edge] (19) to (33);
                    \draw [style=Black Edge] (20) to (33);
                    \draw [style=Black Edge] (21) to (54);
                    \draw [style=Black Edge] (22) to (28);
                    \draw [style=Black Edge] (23) to (34);
                    \draw [style=Black Edge] (23) to (59);
                    \draw [style=Black Edge] (24) to (34);
                    \draw [style=Black Edge] (24) to (48);
                    \draw [style=Black Edge] (25) to (52);
                    \draw [style=Black Edge] (25) to (53);
                    \draw [style=Black Edge] (26) to (27);
                    \draw [style=Black Edge] (26) to (41);
                    \draw [style=Black Edge] (28) to (54);
                    \draw [style=Black Edge] (29) to (40);
                    \draw [style=Black Edge] (29) to (44);
                    \draw [style=Black Edge] (29) to (46);
                    \draw [style=Black Edge] (29) to (56);
                    \draw [style=Black Edge] (30) to (56);
                    \draw [style=Black Edge] (30) to (57);
                    \draw [style=Black Edge] (31) to (32);
                    \draw [style=Black Edge] (31) to (55);
                    \draw [style=Black Edge] (32) to (33);
                    \draw [style=Black Edge] (32) to (47);
                    \draw [style=Black Edge] (35) to (49);
                    \draw [style=Black Edge] (36) to (37);
                    \draw [style=Black Edge] (36) to (38);
                    \draw [style=Black Edge] (36) to (48);
                    \draw [style=Black Edge] (36) to (54);
                    \draw [style=Black Edge] (36) to (59);
                    \draw [style=Black Edge] (37) to (39);
                    \draw [style=Black Edge] (37) to (40);
                    \draw [style=Black Edge] (38) to (41);
                    \draw [style=Black Edge] (38) to (49);
                    \draw [style=Black Edge] (38) to (59);
                    \draw [style=Black Edge] (39) to (41);
                    \draw [style=Black Edge] (39) to (44);
                    \draw [style=Black Edge] (39) to (45);
                    \draw [style=Black Edge] (41) to (42);
                    \draw [style=Black Edge] (41) to (43);
                    \draw [style=Black Edge] (42) to (55);
                    \draw [style=Black Edge] (43) to (55);
                    \draw [style=Black Edge] (45) to (46);
                    \draw [style=Black Edge] (46) to (47);
                    \draw [style=Black Edge] (46) to (52);
                    \draw [style=Black Edge] (46) to (55);
                    \draw [style=Black Edge] (46) to (56);
                    \draw [style=Black Edge] (47) to (53);
                    \draw [style=Black Edge] (47) to (58);
                    \draw [style=Black Edge] (50) to (51);
                    \draw [style=Black Edge] (50) to (53);
                    \draw [style=Black Edge] (51) to (52);
                    \draw [style=Black Edge] (54) to (59);
                    \draw [style=Black Edge] (56) to (57);
                \end{tikzpicture}
            }
        	\caption{The power grid described in \cite{C2020}, which we will call $G$. This graph has 60 vertices and has power domination number equal to 11. The dark blue vertices constitute a minimum power dominating set. Table~\ref{tab:betabestsets} shows the maximum observance for each size from $0$ to $11$, along with the relevant cost functions and intervals on which the given size is useful.}
        	\label{fig:60nodesystem}
        \end{figure}
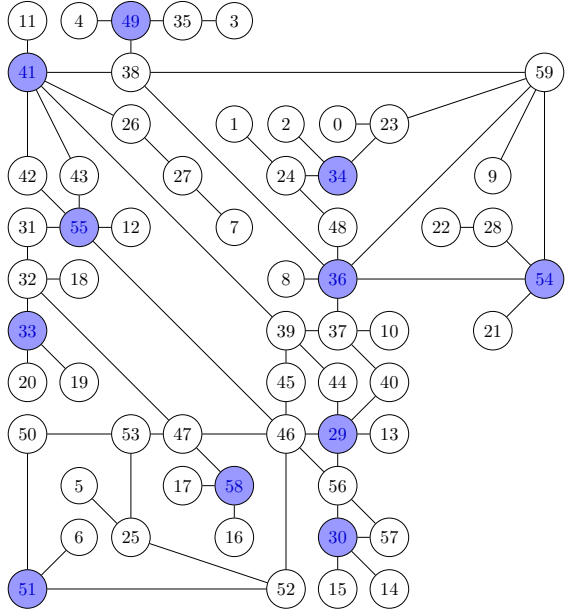
        
        \begin{table}
        	\centering
        	\begin{tabular}{|l|l|l|l|l|}
        		\hline
        		$k$&$\mathrm{maxObs}(G;k)$&Set of size $k$ achieving max observance&$\mathrm{C}(G;k,\beta)$&Useful Interval\\ 
        		\hline
        		0&0&$\emptyset$&$60\beta$& $\left[0,\frac{1}{10}\right]$\\ 
        		\hline
        		1&10&$\{41\}$&$1 + 50\beta$     & $\left[\frac{1}{10},\frac{3}{25}\right]$\\ 
        		\hline
        		2&18&$\{29, 41\}$&$2 + 42\beta$     & N.A. \\ 
        		\hline
        		3&26&$\{25, 29, 41\}$&$3 + 34\beta$     & N.A.\\ 
        		\hline
        		4&35&$\{30, 36, 51, 55\}$&$4 + 25\beta$     & $\left[\frac{3}{25},\frac{1}{6}\right]$\\ 
        		\hline
        		5&41&$\{0, 30, 36, 51, 55\}$&$5 + 19\beta$     & N.A.\\ 
        		\hline
        		6&47&$\{0, 25, 30, 32, 36, 41\}$&$6 + 13\beta$     & $\left[\frac{1}{6},\frac{1}{3}\right]$ \\
                \hline
        		7&50&$\{0, 3, 25, 30, 32, 36, 41\}$&$7 + 10\beta$     &N.A.\\ 
        		\hline
        		8&53&$\{0, 1, 5, 14, 29, 32, 41, 54\}$&$8 + 7\beta$      &N.A.\\ 
        		\hline
        		9&56&$\{23, 25, 29, 30, 32, 36, 41, 49, 54\}$&$9 + 4\beta$& $\left[\frac{1}{3},\frac{1}{2}\right]$\\ 
        		\hline
        		10&58&$\{29, 30, 33, 34, 36, 41, 49, 51, 54, 55\}$& $10 + 2\beta$&N.A.\\ 
        		\hline
        		11&60&$\{29, 30, 33, 34, 36, 41, 49, 51, 54, 55, 58\}$ & $11$& $\left[\frac{1}{2},\infty\right)$\\ 
        		\hline
        	\end{tabular}
        	\caption{Maximum observance, a representative set, the cost function, and the useful interval for each size $k$ from $0$ to $11$ for the graph in Figure~\ref{fig:60nodesystem}.}
        	\label{tab:betabestsets}
        \end{table}
        
        One can theoretically use the information in Table~\ref{tab:betabestsets} to make informed decisions about power grid observance.
        For example, one would never want to use exactly $5$ PMUs to monitor this power grid: the extra vertices observed as compared to using $4$ PMUs is not outweighed by the additional cost until it's optimal to use $6$ PMUs.
        Another example would be if we had an estimate for $\beta$.
        Say, perhaps, we believe $\beta\approx 0.25$ (i.e. the cost of a PMU is approximately $4$ times the cost of not observing any given vertex in the graph).
        In this case, our analysis suggests that using exactly $6$ sensors will minimize the total cost of PMUs and the costs associated with not observing the entire power grid.
        
        \begin{figure}
        	\centering
            \includegraphics{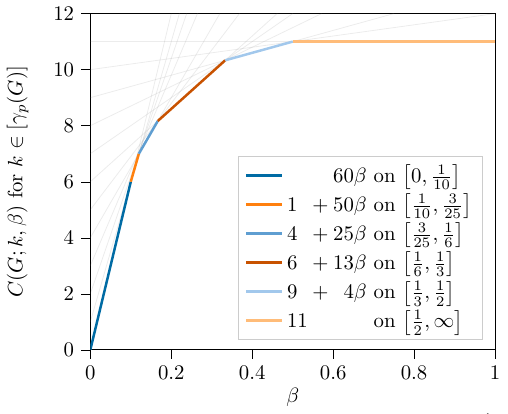}
            \caption{A plot of the cost functions from Table~\ref{tab:betabestsets}. Segments of functions on intervals where they are $\beta$-best are indicated, and all other segments and functions are displayed in gray.}
            \label{figure plot of cost functions for 60 node system}
        \end{figure}
    
    \subsection{Main results}
        Here we summarize the main results of our paper.
        Our first main theorem shows that essentially any collection of sizes can be the set of all useful sizes for a graph, the only condition being that $0$ and $\gamma_P(G)$ are necessarily included in the collection of desired sizes.
        The proof of Theorem~\ref{theorem every set realizable} appears in Section~\ref{subsection final proof grid gadgets}.

        \begin{customthm}{\ref{theorem every set realizable}}
            Let $s\in \mathbb{N}$.
            Let $R\subseteq \{0,1,\dots,s\}$ be a set such that $0,s\in R$.
            Then there exists a graph $G$ such that the set of useful sizes of $G$ is exactly $R$.
        \end{customthm}
        
        Our second main result concerns the situation where the cost function is minimized while using $\gamma_P(G)$ sensors.
        A \emph{fort} in a graph $G$ is a subset $F\subseteq V(G)$ with the property that there does not exist a vertex $x\in V(G)\setminus F$ with a exactly one neighbor in $F$.
        Forts were first introduced by Fast and Hicks~\cite{FH2018} in the context of zero-forcing.
        We define the \emph{minimum fort number} of the graph $G$, $\underline{f}(G)$ to be the size of the smallest fort in $G$.
        Based on the minimum fort number, we can give a large range of $\beta$ for which $\gamma_P$ is useful.
        The proof of Theorem~\ref{theorem when to use gamma sensors} appears in Section~\ref{subsection using gamma sensors}.
        
        \begin{customthm}{\ref{theorem when to use gamma sensors}}
            Let $G$ be a graph on $n$ vertices. 
            Let 
            \[
            B:=\begin{cases}
                1&\text{ if }\underline{f}(G)=1,\\
                \frac{1}{2}&\text{ if }\underline{f}(G)=2,\\
                \frac{1}{3}-\frac{\underline{f}(G)-3}{3(\underline{f}(G)+3\gamma_P(G)-6)}&\text{ if }\underline{f}(G)\geq 3.
            \end{cases}
            \]
            If $\displaystyle \beta\geq \max\left\{B,\frac{\gamma_P(G)}{n}\right\}$, then any minimum power dominating set is $\beta$-best.
            Furthermore, this is best-possible if $\underline{f}(G)\leq 3$.
        \end{customthm}
        
        In addition to these two results, we study marginal cost and marginal observance; i.e. the change in cost/observance when choosing to use one extra sensor, and provide some results on the relationship between marginal observance and useful sizes. 
    
    \subsection{Definitions, notation, and organization}
    
        The \emph{entrance} of the fort $F$, denoted $\mathfrak{e}(F)$ is defined to be the set of vertices in $V(G)\setminus F$ that have a neighbor in $F$.
        We also let $\mathfrak{t}(F):=F\cup \mathfrak{e}(F)$ denote the fort along with it's entrance.
        
        Given two graphs $G$ and $H$, the Cartesian product $G\square H$ is the graph with vertex set $V(G)\times V(H)$, and $(i,j)$ is adjacent to $(i',j')$ in $G\square H$ if and only if either both $i=i'$ and $jj'\in E(H)$ or both $j=j'$ and $ii'\in E(G)$.
        
        Given some function $f$ with domain $D$, we write $\displaystyle\argmax_{d\in D} f(d)$ to denote the set of all maximizers of $f$, i.e. values $d^*\in D$ such that $f(d^*)=\displaystyle\max_{d\in D} f(d)$.
        
        The rest of the paper is organized as follows.
        In Section~\ref{section useful sizes} we prove Theorem~\ref{theorem every set realizable}.
        In Section~\ref{section marginal cost}, we introduce the ideas of marginal cost and marginal observance, and then use these ideas to prove Theorem~\ref{theorem when to use gamma sensors}.
        In Section~\ref{section concluding remarks}, we provide some open questions and concluding remarks.

\section{Useful sizes}\label{section useful sizes}
    Our main goal in this section is to prove Theorem~\ref{theorem every set realizable}.
    To do this, we will construct a graph $H$ with the desired set of useful sizes, say $R$, in mind.
    We start constructing $H$ with a clique $K$ of size $\max(R)$.
    For each $i\in R\setminus \{0\}$ we will affix a ``gadget'' to $i$ vertices of $K$.
    Each gadget will only be fully observed if each vertex of the associated set of $i$ vertices in $K$ has a sensor; otherwise, only a relatively small amount of the gadget will be observed.
    The sizes of the gadgets start out very large, and decrease as $i$ increases.
    The rate at which the gadget size decreases will be chosen to ensure that $i$ is a useful size.
    For each $j\not\in R$, there will be no corresponding gadget and thereby $j$ will not be a useful size.
    The explicit construction of these gadgets is described in Section~\ref{subsection grid gadget}.

    The gadgets are based off of cylindrical grid graphs, i.e., the Cartesian product $C_a\Box P_b$. 
    In Section~\ref{subsection bounding grid gadgets}, we will bound how much of these graphs can be observed if there are not sensors placed on a certain set.
    More specifically, if $a$ is much larger than $b$, it was shown in \cite{DH2006} that $\gamma_P(C_a\Box P_b)$ is about $b/4$.
    We will prove that if fewer than $t/4$ sensors are placed on $C_a\Box P_b$, then  the maximum amount observed depends only on $b$, not on $a$.
    This key fact allows our gadgets to have the desired property that if we do not have sensors on a specific set, we do not observe many vertices.
    
    To complete the proof, in Section~\ref{subsection final proof grid gadgets} we establish Lemma~\ref{lemma characterization of beta good sizes}, which gives an efficient characterization of which sizes are useful based on the values of $\mathrm{maxObs}(G;k)$ for $k\in [\gamma_P(G)]$ for any graph $G$.
    Then, we formally describe the construction of $H$, and the rest of the proof is dedicated to establishing crude, but good-enough bounds on $\mathrm{maxObs}(H;k)$ to apply Lemma~\ref{lemma characterization of beta good sizes}.
    
    \subsection{Grid gadgets}\label{subsection grid gadget}
        We now describe the main constructive tool that allows us to control which sizes are useful.
        
        \begin{construction}
            We construct the $\boxplus^{a}_{\ell,2m}$-gadget $G'$ by performing the following steps.
            \begin{enumerate}
                \item 
                    Start with the cylindrical grid $P_{\ell} \square C_{2m}$ with vertex set $[\ell] \times [2m]$ and the usual edge set.
                    Call these $2m\ell$ vertices in this $P_{\ell} \square C_{2m}$ the \emph{grid vertices}.
                \item 
                    Append a leaf to each grid vertex of the form $(t,1)$ and $(t,2)$ for each $t \in [\ell]$.
                    This creates $2\ell$ new vertices we call $L_1$.
                \item 
                    Append a leaf to each vertex in $L_1$.
                    This creates another $2\ell$ new vertices we call $L_2$.
                \item 
                    Add $a$ new vertices labeled as $x_1,\dots,x_a$.
                    We call these the \emph{affix vertices}.
                \item 
                    Add the edge $yx_i$ for each $y \in L_2$ and for each affix vertex $x_i$.
                    Then subdivide $yx_i$, which creates $2a\ell$ vertices with degree 2 we call $L_3$. 
            \end{enumerate}
            The resulting gadget $G'$ has vertices cleanly delineated as affix vertices, grid vertices, or members of $L_1, L_2,$ or $L_3$.
        \end{construction}
        We also borrow the affix operation from~\cite{BEKV2025}.
        \begin{construction}[Affix Gadget Operation, $G = \glue{G_0}{A}{G'}$]\label{construction affix gadget operation}
            Given a graph $G_0$, $A\subseteq V(G_0)$, and a gadget $G'=\boxplus^{|A|}_{\ell,k}$ we construct the graph $G = G_0\boxminus_A G'$ by identifying each $w\in A$ with a unique affix vertex of $G'$.
            We say that the gadget $G'$ has been \emph{affixed} to the graph $G_0$ to create the graph $G$.
        \end{construction}
        Technically the above construction is not uniquely defined---there are many ways in which one could identify the vertices in $A$ with affix vertices.
        For our purposes, the specific identification is irrelevant, so we allow an arbitrary choice of identification.
        Figure~\ref{figure P4 grid gadget} shows the construction of the gadget $G'=\boxplus^{3}_{4,6}$ and affixes $G'$ to $G_0=K_4$ at $\{u_1, u_2, u_3\}$.
        
        \begin{figure}[htbp]
            \centering
            \begin{tikzpicture}
                \node [style=Small Black Node] (0) at (8.75, -0.5) {};
                \node [style=Small Black Node] (1) at (9.25, -1) {};
                \node [style=Small Black Node] (2) at (8.5, -1) {};
                \node [style=Small Black Node] (3) at (7.75, -1) {};
                \node [style=Small Black Node] (5) at (8, -0.5) {};
                \node [style=Small Black Node] (8) at (7.25, -0.5) {};
                \node [style=Small Black Node] (10) at (8.75, -2) {};
                \node [style=Small Black Node] (11) at (9.25, -2.5) {};
                \node [style=Small Black Node] (12) at (8.5, -2.5) {};
                \node [style=Small Black Node] (13) at (7.75, -2.5) {};
                \node [style=Small Black Node] (15) at (8, -2) {};
                \node [style=Small Black Node] (18) at (7.25, -2) {};
                \node [style=Small Black Node] (20) at (8.75, 1) {};
                \node [style=Small Black Node] (21) at (9.25, 0.5) {};
                \node [style=Small Black Node] (22) at (8.5, 0.5) {};
                \node [style=Small Black Node] (23) at (7.75, 0.5) {};
                \node [style=Small Black Node] (25) at (8, 1) {};
                \node [style=Small Black Node] (28) at (7.25, 1) {};
                \node [style=Small Green Node] (29) at (6.25, -1) {};
                \node [style=Small Green Node] (30) at (6.25, -0.5) {};
                \node [style=Small Green Node] (31) at (6.25, -2.5) {};
                \node [style=Small Green Node] (32) at (6.25, -2) {};
                \node [style=Small Green Node] (33) at (6.25, 0.5) {};
                \node [style=Small Green Node] (34) at (6.25, 1) {};
                \node [style=Small Red Node] (35) at (5.25, -1) {};
                \node [style=Small Red Node] (36) at (5.25, -0.5) {};
                \node [style=Small Red Node] (37) at (5.25, -2.5) {};
                \node [style=Small Red Node] (38) at (5.25, -2) {};
                \node [style=Small Red Node] (39) at (5.25, 0.5) {};
                \node [style=Small Red Node] (40) at (5.25, 1) {};
                \node [style=Small Black Node] (41) at (8.75, 2.5) {};
                \node [style=Small Black Node] (42) at (9.25, 2) {};
                \node [style=Small Black Node] (43) at (8.5, 2) {};
                \node [style=Small Black Node] (44) at (7.75, 2) {};
                \node [style=Small Black Node] (45) at (8, 2.5) {};
                \node [style=Small Black Node] (48) at (7.25, 2.5) {};
                \node [style=Small Green Node] (49) at (6.25, 2) {};
                \node [style=Small Green Node] (50) at (6.25, 2.5) {};
                \node [style=Small Red Node] (51) at (5.25, 2) {};
                \node [style=Small Red Node] (52) at (5.25, 2.5) {};
                \node [style=Small Black Node] (54) at (0, 0) {};
                \node [style=Small Black Node] (55) at (0, 5) {};
                \node [style=Small Blue Node] (56) at (2, 3) {};
                \node [style=Small Blue Node] (57) at (2.5, 3) {};
                \node [style=Small Blue Node] (58) at (1, 3) {};
                \node [style=Small Blue Node] (59) at (1.5, 3) {};
                \node [style=Small Blue Node] (60) at (3, 3) {};
                \node [style=Small Blue Node] (61) at (3.5, 3) {};
                \node [style=Small Blue Node] (62) at (4, 3) {};
                \node [style=Small Blue Node] (63) at (4.5, 3) {};
                \node [style=Small Blue Node] (64) at (1, -0.75) {};
                \node [style=Small Blue Node] (65) at (1, -0.25) {};
                \node [style=Small Blue Node] (66) at (1, -1.75) {};
                \node [style=Small Blue Node] (67) at (1, -1.25) {};
                \node [style=Small Blue Node] (68) at (1, 0.25) {};
                \node [style=Small Blue Node] (69) at (1, 0.75) {};
                \node [style=Small Blue Node] (70) at (1, 1.25) {};
                \node [style=Small Blue Node] (71) at (1, 1.75) {};
                \node [style=Small Black Node] (72) at (0, -5.5) {};
                \node [style=Small Blue Node] (73) at (2, -3) {};
                \node [style=Small Blue Node] (74) at (2.5, -3) {};
                \node [style=Small Blue Node] (75) at (1, -3) {};
                \node [style=Small Blue Node] (76) at (1.5, -3) {};
                \node [style=Small Blue Node] (77) at (3, -3) {};
                \node [style=Small Blue Node] (78) at (3.5, -3) {};
                \node [style=Small Blue Node] (79) at (4, -3) {};
                \node [style=Small Blue Node] (80) at (4.5, -3) {};
                \node [style=Small Black Node] (81) at (-1.5, 0) {};
                \node [style=Text Node] (82) at (-0.55, 0.25) {$u_2$};
                \node [style=Text Node] (83) at (-0.55, 5.25) {$u_1$};
                \node [style=Text Node] (84) at (-0.55, -5.25) {$u_3$};
                \node [style=Text Node] (85) at (-2.05, 0.25) {$u_4$};
                \draw [style=Black Edge] (10) to (11);
                \draw [style=Black Edge] (11) to (12);
                \draw [style=Black Edge] (12) to (13);
                \draw [style=Black Edge] (13) to (18);
                \draw [style=Black Edge] (15) to (10);
                \draw [style=Black Edge] (18) to (8);
                \draw [style=Black Edge] (8) to (28);
                \draw [style=Black Edge] (25) to (20);
                \draw [style=Black Edge] (20) to (21);
                \draw [style=Black Edge] (21) to (22);
                \draw [style=Black Edge] (22) to (23);
                \draw [style=Black Edge] (23) to (28);
                \draw [style=Black Edge] (5) to (0);
                \draw [style=Black Edge] (0) to (1);
                \draw [style=Black Edge] (1) to (2);
                \draw [style=Black Edge] (2) to (3);
                \draw [style=Black Edge] (3) to (8);
                \draw [style=Black Edge] (3) to (13);
                \draw [style=Black Edge] (22) to (2);
                \draw [style=Black Edge] (2) to (12);
                \draw [style=Black Edge] (3) to (23);
                \draw [style=Black Edge] (11) to (1);
                \draw [style=Black Edge] (1) to (21);
                \draw [style=Black Edge] (25) to (5);
                \draw [style=Black Edge] (5) to (15);
                \draw [style=Black Edge] (10) to (0);
                \draw [style=Black Edge] (0) to (20);
                \draw [style=Black Edge] (28) to (34);
                \draw [style=Black Edge] (33) to (23);
                \draw [style=Black Edge] (8) to (30);
                \draw [style=Black Edge] (29) to (3);
                \draw [style=Black Edge] (18) to (32);
                \draw [style=Black Edge] (31) to (13);
                \draw [style=Black Edge] (34) to (40);
                \draw [style=Black Edge] (39) to (33);
                \draw [style=Black Edge] (30) to (36);
                \draw [style=Black Edge] (35) to (29);
                \draw [style=Black Edge] (32) to (38);
                \draw [style=Black Edge] (31) to (37);
                \draw [style=Black Edge] (45) to (41);
                \draw [style=Black Edge] (41) to (42);
                \draw [style=Black Edge] (42) to (43);
                \draw [style=Black Edge] (43) to (44);
                \draw [style=Black Edge] (44) to (48);
                \draw [style=Black Edge] (48) to (50);
                \draw [style=Black Edge] (49) to (44);
                \draw [style=Black Edge] (50) to (52);
                \draw [style=Black Edge] (51) to (49);
                \draw [style=Black Edge] (41) to (20);
                \draw [style=Black Edge] (25) to (45);
                \draw [style=Black Edge] (42) to (21);
                \draw [style=Black Edge] (43) to (22);
                \draw [style=Black Edge] (44) to (23);
                \draw [style=Black Edge] (28) to (48);
                \draw [style=Black Edge] (18) to (15);
                \draw [style=Black Edge] (5) to (8);
                \draw [style=Black Edge] (28) to (25);
                \draw [style=Black Edge] (45) to (48);
                \draw [style=Black Edge] (66) to (54);
                \draw [style=Black Edge] (54) to (67);
                \draw [style=Black Edge] (64) to (54);
                \draw [style=Black Edge] (54) to (65);
                \draw [style=Black Edge] (68) to (54);
                \draw [style=Black Edge] (54) to (69);
                \draw [style=Black Edge] (70) to (54);
                \draw [style=Black Edge] (54) to (71);
                \draw [style=Black Edge] (58) to (55);
                \draw [style=Black Edge] (55) to (59);
                \draw [style=Black Edge] (56) to (55);
                \draw [style=Black Edge] (55) to (57);
                \draw [style=Black Edge] (60) to (55);
                \draw [style=Black Edge] (55) to (61);
                \draw [style=Black Edge] (62) to (55);
                \draw [style=Black Edge] (55) to (63);
                \draw [style=Black Edge] (63) to (52);
                \draw [style=Black Edge] (52) to (71);
                \draw [style=Black Edge] (70) to (51);
                \draw [style=Black Edge] (51) to (62);
                \draw [style=Black Edge] (61) to (40);
                \draw [style=Black Edge] (40) to (69);
                \draw [style=Black Edge] (68) to (39);
                \draw [style=Black Edge] (39) to (60);
                \draw [style=Black Edge] (57) to (36);
                \draw [style=Black Edge] (36) to (65);
                \draw [style=Black Edge] (64) to (35);
                \draw [style=Black Edge] (35) to (56);
                \draw [style=Black Edge] (59) to (38);
                \draw [style=Black Edge] (38) to (67);
                \draw [style=Black Edge] (66) to (37);
                \draw [style=Black Edge] (37) to (58);
                \draw [style=Black Edge] (75) to (72);
                \draw [style=Black Edge] (72) to (76);
                \draw [style=Black Edge] (73) to (72);
                \draw [style=Black Edge] (72) to (74);
                \draw [style=Black Edge] (77) to (72);
                \draw [style=Black Edge] (72) to (78);
                \draw [style=Black Edge] (79) to (72);
                \draw [style=Black Edge] (72) to (80);
                \draw [style=Black Edge] (80) to (37);
                \draw [style=Black Edge] (38) to (79);
                \draw [style=Black Edge] (78) to (35);
                \draw [style=Black Edge] (36) to (77);
                \draw [style=Black Edge] (74) to (39);
                \draw [style=Black Edge] (40) to (73);
                \draw [style=Black Edge] (76) to (51);
                \draw [style=Black Edge] (52) to (75);
                \draw [style=Black Edge] (81) to (72);
                \draw [style=Black Edge] (55) to (54);
                \draw [style=Black Edge] (54) to (72);
                \draw [style=Black Edge] (81) to (55);
                \draw [style=Black Edge] (81) to (54);
                \node [style=Text Node] (86) at (-1.25, 5.75) {\emph{Affix vertices}};
                \draw[rounded corners=10pt, style=Dashed Edge] (-0.75, -5.75) rectangle (0.5, 5.5);
                \node [style=Text Node] (86) at (4.35, 3.5) {$L_3$};
                \draw[rounded corners=10pt, style=Dashed Edge] (0.75, -3.25) rectangle (4.75, 3.25);
                \node [style=Text Node] (87) at (5, 3) {$L_2$};
                \draw[rounded corners=10pt, style=Dashed Edge] (4.875, -2.75) rectangle (5.625, 2.75);
                \node [style=Text Node] (88) at (6, 3) {$L_1$};
                \draw[rounded corners=10pt, style=Dashed Edge] (5.875, -2.75) rectangle (6.625, 2.75);
                \node [style=Text Node] (88) at (7, 3) {\emph{Grid vertices}};
                \draw[rounded corners=10pt, style=Dashed Edge] (7, -2.75) rectangle (9.5, 2.75);
            \end{tikzpicture}
            \caption{The graph $\glue{K_4}{\{u_1,u_2,u_3\}}{\boxplus^{3}_{4,6}}$}
            \label{figure P4 grid gadget}
        \end{figure}
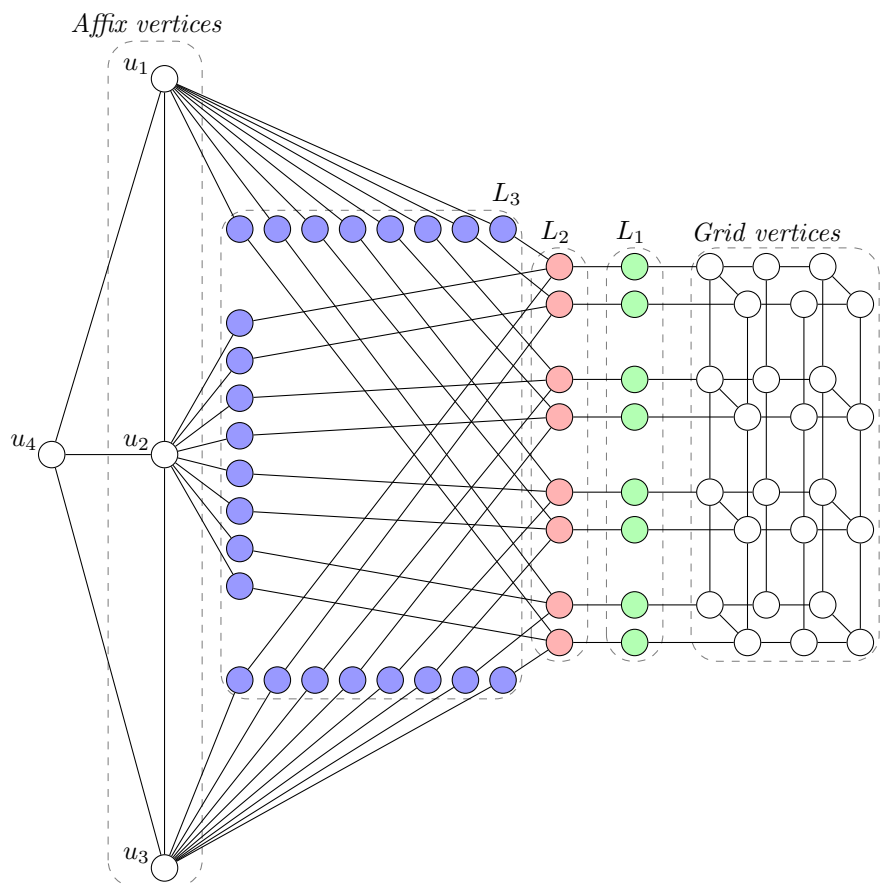

    \subsection{Bounding the observance on a grid gadget}\label{subsection bounding grid gadgets}
        
        Borrowing notation from~\cite{DH2006}, define the following.
        Given a cylindrical grid graph $P_\ell\square C_{2m}$ with vertex set $[\ell]\times [2m]$, we will call sets of the form $[\ell]\times \{k\}$ for some fixed $k\in [2m]$ a \emph{column}.
        Let \[U_1=\{(x,y)\in [\ell]\times [2m]\mid x+y\text{ is odd}\}\text{, and }U_2=([\ell]\times [2m])\setminus U_1.\]
        The sets $U_1,U_2$ form a bipartition of $V(P_\ell\square C_{2m})$.

        We also adapt the following definitions from~\cite{DH2006}. 
        Given a graph $G$ and a set $T\subseteq V(G)$, the \emph{zero forcing closure} $C_G(T)$ is defined recursively: for any vertex $x$ \textbf{in the set} $T$, if $|N(x)\setminus T|=1$ then we add the single vertex in $N(x)\setminus T$ to $T$ and repeat.
        Note that $\mathrm{Obs}(G;S)=C_G(N[S])$.
        As this process is recursive, we can define a \emph{chronological list of forces}, denoted $\mathcal{F}(T) =(x_1\to y_1, x_2\to y_2,\dots,x_k\to y_k)$, to be a (not necessarily unique) sequence where the vertex $x_i$ causes the vertex $y_i$ to be added to $T$, resulting in $C_G(T)=T\cup \bigcup_{i=1}^k \{y_i\}.$
        A similar parameter, the \emph{star closure} $C^*_G(T)$ of $T$ is also defined recursively: for any vertex $x$ \textbf{in the graph} $G$ if $|N(x)\setminus T|=1$ then we add the single vertex in $N(x)\setminus T$ to $T$ and repeat.
        We note that the star closure is also the skew derived set as in \cite{IMA2010}, with connections to skew zero forcing.

        We now establish several lemmas which will be key to the proof of Theorem~\ref{theorem every set realizable}.
        Our first lemma will allow us to bound the size of our observed set based on the star closure.
        \begin{lemma}\label{lemma forcing the even and odd parts of a bipartite graph}
            Let $G$ be a bipartite graph with partite sets $V_1$ and $V_2$ and $\delta(G)\geq 2$. Given a set $T\subseteq V(G)$, we have
            \[ C_G(T)\subseteq C_G^*(T\cap V_1)\cup C_{G}^*(T\cap V_2). \]
        \end{lemma}
        \begin{proof}
            First we claim that $C_G^*(T\cap V_1)\subseteq V_1$.
            Towards a contradiction, assume that when starting the star forcing process from $T\cap V_1$, at some point a vertex from $V_2$ is forced, say $y\in V_2$. We may further assume without loss of generality that actually $y$ is the first vertex forced throughout the entire forcing process. Then there exists some $x\in V(G)$ with  $ N(x)\setminus T=\{y\}$.
            Since $G$ is bipartite and $x$ has a neighbor in $V_2$, $N(x)\subseteq V_2$.
            As $\delta(G) \geq 2$, $x$ must have at least one more neighbor, $z$, and $z\in V_2$.
            However, then we have $\{y,z\}\subseteq N(x) \setminus (T\cap V_1)$, so $x$ cannot force $y$, a contradiction.
            Similarly, we have $C_G^*(T\cap V_2)\subseteq V_2$.

            Now consider a chronological list of forces, $\mathcal{F}(T) =(x_1\to y_1, x_2\to y_2,\dots,x_k\to y_k)$ with $C_G(T)=T\cup \bigcup_{i=1}^k \{y_i\}.$
            Towards a contradiction, let $\tilde{i}$ be the smallest index such that $y_{\tilde{i}}\not\in C_G^*(T\cap V_1)\cup C_{G}^*(T\cap V_2)$.
            Assume without loss of generality that $y_{\tilde{i}}\in V_1$.
            Then $x_{\tilde{i}}\in V_2$ and $N(x_{\tilde{i}})\subseteq V_1$.
            We note that since $x_{\tilde{i}}$ was able to force $y_{\tilde{i}}$, we have that 
                \[ N(x_{\tilde{i}})\setminus \{y_{\tilde{i}}\}\subseteq T\cup \{y_1,y_2,\dots,y_{\tilde{i}-1}\}\subseteq C_G^*(T\cap V_1)\cup C_{G}^*(T\cap V_2). \]
            In particular, since $N(x_{\tilde{i}})\subseteq V_2$ and $C_{G}^*(T\cap V_1)\cap V_2=\emptyset$, we have $N(x_{\tilde{i}})\setminus \{y_{\tilde{i}}\}\subseteq C_{G}^*(T\cap V_2)$.
            However, by the definition of star closure, we should have had $y_{\tilde{i}}\in C_{G}^*(T\cap V_2)$, a contradiction.
            Thus,
                \[ C_G(T)\subseteq C_G^*(T\cap V_1)\cup C_{G}^*(T\cap V_2). \qedhere \]
        \end{proof}

        The following lemma will help us to estimate the observed set in a grid gadget $\boxplus^a_{\ell,2m}$ based on prior known work about power domination on the cylindrical grid graph $P_{\ell}\square C_{2m}$ from~\cite{DH2006}.
        In our application, the set $X$ below will be the grid vertices in a gadget, while the set $Y$ will be a subset of the vertices in the first two columns of the cylindrical grid graph, i.e. those which connect outside the cylindrical grid graph. 

        \begin{lemma}\label{lemma forcing on an induced subgraph}
            Let $G$ be a graph, and $X, F\subseteq V(G)$. 
            Let $\mathcal{F}(F) =(u_1\to v_1, u_2\to v_2,\dots,u_k\to v_k)$ with $\displaystyle C_G(F)=F\cup\bigcup_{i=1}^k \{v_i\}$ be a chronological list of forces.
            Let $Y\subseteq X$ be a collection of vertices such that if $u\to v\in \mathcal{F}(F)$ is such that $u\in V(G)\setminus X$ and $v\in X$, then $v\in Y$.
            Then, 
            \[ C_G(F)\cap X\subseteq C_{G[X]}((F\cap X)\cup Y) \]
        \end{lemma}
        \begin{proof}
            Towards a contradiction, let $\tilde{i}$ denote the smallest index in $\mathcal{F}(F)$ such that 
                \[ \displaystyle v_{\tilde{i}}\in X\setminus C_{G[X]}((F\cap X)\cup Y). \]
           Thus, $u_{\tilde{i}}\not\in Y$ which implies that $v_{\tilde{i}}\in X$. 
            Furthermore,
                \[ N_G[u_{\tilde{i}}]\setminus \{v_{\tilde{i}}\}\subseteq F\cup \{v_1,v_2,\dots,v_{\tilde{i}-1}\}. \]
            Observe that $N_{G[X]}[u_{\tilde{i}}]\setminus \{v_{\tilde{i}}\}=(N_G[u_{\tilde{i}}]\setminus \{v_{\tilde{i}}\})\cap X$. 
            Then, 
                \[ (N_G[u_{\tilde{i}}]\setminus \{v_{\tilde{i}}\})\cap X\subseteq (F\cup \{v_1,v_2,\dots,v_{\tilde{i}-1}\})\cap X\subseteq C_{G[X]}((F\cap X)\cup Y), \]
            where the last relation follows by minimality of $\tilde{i}$.
            However, this means that in $G[X]$, $u_{\tilde{i}}$ can force $v_{\tilde{i}}$ but $v_{\tilde{i}} \not \in C_{G[X]}((F\cap X)\cup Y)$, a contradiction.
        \end{proof}

        We also need the following result from~\cite{DH2006}, which we use to bound the number of observed grid vertices.
        \begin{lemma}[Lemma 2 from \cite{DH2006}]\label{lemma star cover in cylider intersects columns}
            Let $m\geq n\geq 2$ and let $G=P_n\square C_{2m}$ with partite sets $U_1$ and $U_2$.
            If $T\subseteq U_i$ for some $i\in \{1,2\}$ and $|T|<n$, then $C^*(T)$ contains vertices from at most $|T|$ columns.
        \end{lemma}
        
        The final tool we need is the following lemma, which bounds how much of a $\boxplus^{a}_{\ell,2m}$-gadget any subset $S\subseteq V(G)$ can observe if $S$ does not include the entire affix set.
        The specific bound provided below is not particularly enlightening, but we remind the reader that the crucial aspect is that the bound is independent of the parameter $m$.
        \begin{lemma}\label{lemma bounded observence in grids}
            Let $s\in \mathbb{N}$ and $G'$ be a graph with $A\subseteq V(G')$ a set such that every vertex in $A$ is adjacent to at least $s+1$ leaves in $V(G')\setminus A$.
            Let $m\geq \ell>4s$.
            Let $G:=G'\boxminus_A\boxplus_{\ell,2m}^{|A|}$.
            Then for any set $S\subseteq V(G)$ with $|S|=s$, if $A\subseteq S$, then every vertex in $V(G)\setminus V(G')$ is observed by $S$.
            On the other hand,
            if $A\not\subseteq S$,
            we have that
            \[ |\mathrm{Obs}(G;S)\cap (V(G)\setminus V(G'))|\leq 2\ell(|A|+3s+2). \]
            
        \end{lemma}
        \begin{proof}
            If $A\subseteq S$, then $L_3\subseteq N[S]$. 
            Then the vertices in $L_3$ force all the vertices in $L_2$, and then the vertices in $L_2$ force all the vertices in $L_1$, and then the vertices in $L_1$ force the vertices in $\{(t,1),(t,2)\mid t\in [\ell]\}$.
            These vertices in turn are able to force all of the grid vertices, so all of $V(G)\setminus V(G')$ is observed.
            
            Now suppose $A\not\subseteq S$.
            Let $X$ denote the set of grid vertices in $\boxplus_{\ell,2m}^{|A|}$ and let $L:=L_1\cup L_2\cup L_3$.
            Furthermore, $|L|=2\ell(|A|+2)$, so
            \begin{equation}\label{equation Obs cap L is at most the size of L}
                |\mathrm{Obs}(G;S)\cap L|\leq |L| = 2\ell(|A|+2).
            \end{equation}
            
            Now, to bound $|\mathrm{Obs}(G;S)\cap X|$, we need the following claim. 
       
            \begin{claim}\label{claim U has property U}
                Let $x$ be a vertex in $\{(t,1),(t,2)\mid t\in [\ell]\}$, let $y\in L_1$ be the unique neighbor of $x$ in $L_1$, and let $z\in L_2$ be the unique neighbor of $y$ in $L_2$.
                If $x$ is either observed by $y$ in the dominating step or forced by $y$, then there exists a vertex of $S$ in $N[z]$. 
            \end{claim}
            \begin{adjustbox}{minipage=\dimexpr.9\linewidth,center}
                \begin{proof}
                    Assume to the contrary that $S\cap N[z]=\emptyset$, and so $y,z\not\in S$.
                    Thus, $y$ forces $x$.
                    Moreover, $y$ must have been forced by $z$ before $y$ could force $x$. 
                    By assumption, $A\not\subseteq S$, so there exists $a\in A\setminus S$.
                    Let $w\in N(a)\cap N(z)$ be the degree $2$ vertex in $L_3$.
                    Note that $w$ must be observed before $y$ is forced by $z$.
                    Since $N[w]=\{w,a,z\}$, and none of these vertices are in $S$, for $w$ to be observed, it must be forced by $a$, but $a$ is adjacent to $s+1$ leaves, at least one of which must not be in $S$, and thus which is not observed.
                    Thus $a$ cannot force $w$, a contradiction.
                \end{proof}
            \end{adjustbox}\medskip
            
            \noindent We now will construct a ``small'' set $T\subseteq X$ that has the property that
                \[ 
                \mathrm{Obs}(G;S)\cap X\subseteq C_{G[X]}(T). 
                \]
             
             First, for $x\in \{(t,1),(t,2)\mid t\in [\ell]\}$, let $z_x$ denote the unique vertex in $L_2$ such that $d(x,z_x)=2$. Then let  $Y\subseteq \{(t,1),(t,2)\mid t\in [\ell]\}$ be the set of all vertices $x\in Y$ such that $N[z_x]$ contains a vertex from $S$, as in Claim~\ref{claim U has property U}.
             Further, each $N[z_x]$ is disjoint, hence $|Y|\leq s$.
             By construction and Claim~\ref{claim U has property U}, any vertex in $\{(t,1),(t,2)\mid t\in [\ell]\}$ that is dominated or forced by a vertex in $V(G)\setminus X$ is in $Y$.
             We now define
                \[ T:=(N[S]\cap X)\cup Y. \]
            By Lemma~\ref{lemma forcing on an induced subgraph},
            \begin{equation}\label{equation Obs cap X is contained inside zero forcing closure of T}
                \mathrm{Obs}(G;S)\cap X=C_G(N[S])\cap X\subseteq C_{G[X]}(T)
            \end{equation}
                
            We now bound $|T|$.
            Note that $|Y|\leq |S|=s$ since the neighborhoods of vertices in $L_2$ are all disjoint.
            Furthermore, since the maximum degree in $G[X]$ is $4$, we have $|N[S]\cap X|\leq 5|S|=5s$, giving us that
                \[ |T|\leq |N[S]\cap X|+|Y|\leq 6s. \]
            Towards applying Lemma~\ref{lemma forcing the even and odd parts of a bipartite graph}, let $U_1$ and $U_2$ be the partite sets of $G[X]\cong P_{\ell}\square C_{2m}$.
            Then, for $i\in \{1,2\}$, by Lemma~\ref{lemma star cover in cylider intersects columns}, $C^*_{G[X]}(T\cap U_i)$ contains vertices from at most $|T\cap U_i|$ columns of $G[X]$, and thus $C^*_{G[X]}(T\cap U_1)\cup C^*_{G[X]}(T\cap U_2)$ contains vertices from at most $|T\cap U_1|+|T\cap U_2|=|T|$ columns.
            Each column contains $\ell$ vertices, so
                \[ |C^*_{G[X]}(T\cap U_1)\cup C^*_{G[X]}(T\cap U_2)|\leq \ell |T|. \]
            By Lemma~\ref{lemma forcing the even and odd parts of a bipartite graph}, this gives us that
                \[ |C_{G[X]}(T)|\leq \ell|T|\leq 6s\ell. \]
            To finish the proof, we use~\eqref{equation Obs cap X is contained inside zero forcing closure of T} and~\eqref{equation Obs cap L is at most the size of L} to write
            \[
                |\mathrm{Obs}(G;S)\cap (V(G)\setminus V(G'))| = |\mathrm{Obs}(G;S)\cap L|+|\mathrm{Obs}(G;S)\cap X| \leq 2\ell(|A|+2)+6s\ell=2\ell(|A|+3s+2),
            \]
            as desired.
        \end{proof}

    \subsection{Proof of \texorpdfstring{Theorem~\ref{theorem every set realizable}}{Theorem 2.7}}\label{subsection final proof grid gadgets}
        The following lemma gives an easy-to-use characteristic of when a certain size is useful.
        \begin{lemma}\label{lemma characterization of beta good sizes}
            Let $G$ be a graph, and let $m_j:=\mathrm{maxObs}(G;j)$ for each $j\in [\gamma_P]\cup\{0\}$.
            For each $i\in [\gamma_P(G)-1]$, the size $i$ is useful if and only if
                \[ 
                \max_{0\leq j<i}\frac{i-j}{m_i-m_j}< \min_{i<j\leq \gamma_P(G)}\frac{j-i}{m_j-m_i}
                 \]
        \end{lemma}
        \begin{proof}
            For each $j$ with $1\leq j\leq \gamma_P(G)$, let $S_j$ be a set with $|S|=j$ and $|\mathrm{Obs}(G;S_j)|=m_j$.
            We can compute
                \begin{equation}\label{eqn cost i minus cost j} \mathrm{C}(G;S_i,\beta)-\mathrm{C}(G;S_j,\beta)=i-j+\beta(m_j-m_i)
                \end{equation}
            If $j>i$, by Observation~\ref{observation more sensors means more observed} $m_j - m_i > 0$, so the right-hand side of \eqref{eqn cost i minus cost j} is non-positive whenever 
                \[ \beta\leq \frac{j-i}{m_j-m_i}. \] 
            Similarly, for $j<i$, we have that $m_j-m_i<0$ by Observation~\ref{observation more sensors means more observed}, so $\mathrm{C}(G;S_i,\beta)-\mathrm{C}(G;S_j,\beta)$ is non-positive as long as
                \[ \beta\geq \frac{j-i}{m_j-m_i}=\frac{i-j}{m_i-m_j}. \] 
            Thus, there exists a range of choices for $\beta$ such that $\mathrm{C}(G;S_i,\beta)$ is the minimum if and only if
                \[ \max_{0\leq j<i}\frac{i-j}{m_i-m_j}< \min_{i<j\leq \gamma_P(G)}\frac{j-i}{m_j-m_i}.\qedhere \] 
        \end{proof}
        
        Lemma~\ref{lemma characterization of beta good sizes} does not cover the case where $i=0$ or $i=\gamma_P(G)$.
        These two sizes are always useful, as shown in Observation~\ref{observation beta small or large}.

        We are now ready to prove Theorem~\ref{theorem every set realizable}.
        
        \begin{theorem}\label{theorem every set realizable}
            Let $s\in \mathbb{N}$.
            Let $R\subseteq \{0,1,\dots,s\}$ be a set such that $0,s\in R$.
            Then there exists a graph $G$ such that the set of useful sizes of $G$ is exactly $R$.
        \end{theorem}
        \begin{proof}
            Let $R':=R\setminus \{0\}$.
            Start with a clique $K_s$ and add $s+1$ leaves to every vertex to form a graph $G_0$.
            Let $K\subseteq V(G_0)$ denote the vertices in the clique $K_s$.
            For each $i\in [s]$, let $A_i\subseteq K$ be a set with $|A_i|=i$, such that $A_i\subseteq A_{i'}$ for all $i\leq i'$.
            For each $i\in [s]$, define the parameter
            \[
                x_i:=\begin{cases}
                    2\cdot\left\lceil\displaystyle\frac{63s^4+(4s^2+s) \displaystyle\sum_{t > i}x_{t}}{4s+1}\right\rceil&\text{ if }i\in R'\\
                    0&\text{ if }i\in [s]\setminus R'.
                \end{cases}
            \] 
            Note that $x_i$ depends only on $x_{t}$ for $t > i$, so $x_i$ is well-defined.
            For all $i\in R'$, we can bound $x_i$ below by the following  expressions, which will be useful later in the proof:
            \begin{equation}\label{equation bound on x_i 1}
                x_i>\frac{96s^4+s(4s+1)\displaystyle\sum_{t>i}x_{t}}{4s+1}
            \end{equation} 
            and
            \begin{equation}\label{equation bound on x_i 2}
                x_i\geq \frac{126s^4}{4s+1}>\frac{63s^4+63s^3}{4s+1}.
            \end{equation} 
            For each $i\in R'$, affix the gadget $\boxplus_{4s+1,x_i}^{i}$ to $G_0$ at $A_i$.
            With an eye on applying Lemma~\ref{lemma bounded observence in grids} (with $m:=x_i/2$ and $\ell:=4s+1$), we note that for $i\in R'$, $x_i/2\geq 4s+1$.
            Call the graph that results from affixing all these gadgets $G$.
            For $s \geq 1$, the size of $V(G)$ is the number of vertices in $K_s$, plus each clique vertex's $s+1$ leaves, and for each gadget the number of $L_1, L_2, L_3$, and grid vertices, giving us
            \begin{align}
                |V(G)|&=s+s(s+1)+\sum_{i\in R'} \Big(2(4s+1) + 2(4s+1) + 2(4s+1)i + (4s+1)x_i\Big)\nonumber\\
                &\leq s^2+2s+\sum_{i=1}^{s}\Big(4(4s+1)\Big)+\sum_{i=1}^{s}\Big(2(4s+1)i\Big) + \sum_{i\in R'}\Big((4s+1)x_i\Big)\nonumber\\
                &= s^2+2s+4s(4s+1)+2(4s+1)\left(\frac{s(s+1)}{2}\right) +\sum_{i\in R'}\Big((4s+1)x_i\Big)\nonumber\\
                &\leq 33s^3 + (4s+1)\sum_{i\in R'}x_i\label{equation bound on n}
            \end{align}
            Furthermore, note that $A_s=K$ is a power dominating set of $G$.
            By Lemma~\ref{lemma bounded observence in grids}, no set smaller than $A_s$ can observe the $s$th gadget, so $\gamma_P(G)=s$.
            
            Now fix some $i\in [s]$.
            We will give crude bounds on
                \[ m_i:=\mathrm{maxObs}(G;i). \]
            
            Let $\displaystyle S_i\in \argmax_{S\in\binom{V(G)}{i}} |\mathrm{Obs}(G;S)|$ so that $m_i=|\mathrm{Obs}(G;S_i)|$.
            We have that $|\mathrm{Obs}(G;S_i)\cap V(G_0)|\leq |V(G_0)|=s^2+2s$.
            For each $j\in R'$, if $A_j\not\subseteq S_i$, we apply Lemma~\ref{lemma bounded observence in grids} (with $A:=A_j$, $s:=i$, and $\ell:=4s+1$) to get that $S_i$ observes at most
                \[ 
                2(4s+1)(j+3i+2)\leq 60s^2 
                \]
            vertices in the $j$th gadget.
            Also, if $A_j\subseteq S_i$, $S_i$ observes all of the $j$th gadget, which has at most
                \[ 
                2(4s+1)(j+2)+(4s+1)x_j
                \leq 60s^2+(4s+1)x_j 
                \]
            vertices.
            We note that if $A_j\subseteq S_i$, then we must have $j\leq i$, so considering the $|R'|\leq s$ total gadgets and summing up each of these overestimates, we have that
            \begin{equation}\label{equation upper bound on m_i}
                m_i\leq (s^2+2s)+s(60s^2)+ \sum_{j\leq i} (4s+1)x_j\leq 63s^3+\sum_{j\leq i} (4s+1)x_j.
            \end{equation}
            We also need a lower bound on $m_i$.
            Note that $A_j\subseteq A_i$ whenever $j\leq i$, $j\in R'$, so by Lemma~\ref{lemma bounded observence in grids}, $A_i$ observes the entire $j$th gadget. In particular, $A_i$ observes the grid vertices of all of the gadgets for $j<i$, giving the crude bound of 
            \begin{equation}\label{equation lower bound on m_i}
                m_i\geq |\mathrm{Obs}(G;A_i)|\geq \sum_{j\leq i}(4s+1)x_j=(4s+1)\sum_{j\leq i}x_j,
            \end{equation}
            
            We now focus on showing that the set of useful sizes of $G$ is exactly $R$.
            By Observation~\ref{observation beta small or large}, we have that $0$ and $\gamma_P(G)$ are both useful.
            Towards applying Lemma~\ref{lemma characterization of beta good sizes}, fix $i\in R'$ and $0\leq j<i$.
            We have using~\eqref{equation upper bound on m_i} and~\eqref{equation lower bound on m_i}, along with the fact that $i-j\leq s$, that
                \begin{align*}
                    \frac{i-j}{m_i-m_j} &\leq \frac{s}{(4s+1)\displaystyle\sum_{t\leq i}x_{t}-\left(63s^3+ (4s+1)\sum_{t\leq j} x_{t}\right)}\\
                    &= \frac{s}{(4s+1)x_i - 63s^3 + (4s+1) \displaystyle\sum_{j < t < i}x_t}\\
                    &\leq \frac{s}{(4s+1)x_i-63s^3}.
                \end{align*}
            The final expression above is independent of $j$, so we have that
            \begin{equation}\label{equation upper bound on max}
                \max_{0\leq j<i}\frac{i-j}{m_i-m_j}\leq\frac{s}{(4s+1)x_i-63s^3}.
            \end{equation}
            If instead we consider $j>i$, using that $m_j\leq |V(G)|$ and that $j-i\geq 1$, along with~\eqref{equation bound on n}, and~\eqref{equation lower bound on m_i}, we find
            \begin{align*}
                \frac{j-i}{m_j-m_i} &\geq \frac{1}{|V(G)|-(4s+1)\displaystyle\sum_{t\leq i}x_{t}}\\
                &\geq \frac{1}{33s^3 + (4s+1)\displaystyle\sum_{i\in R'}x_i-(4s+1)\displaystyle\sum_{t\leq i}x_{t}}\\
                &\geq \frac{1}{33s^3 +(4s+1)\displaystyle\sum_{t>i}x_{t}}.
            \end{align*}
            Again, this is independent of $j$ and so 
            \begin{equation}\label{equation lower bound on min}
                \frac{1}{33s^3+(4s+1)\displaystyle\sum_{t>i}x_{t}}\leq \min_{i<j\leq \gamma_P(G)}\frac{j-i}{m_j-m_i}.
            \end{equation}
            Using~\eqref{equation bound on x_i 1} to bound $x_i$ in combination with~\eqref{equation lower bound on min}, and then utilizing~\eqref{equation upper bound on max}, we find
            \begin{align*}
                \max_{0\leq j<i}\frac{i-j}{m_i-m_j} \leq\frac{s}{(4s+1)x_i-63s^3}&<\frac{s}{(4s+1)\left( \frac{96s^4+s(4s+1)\displaystyle\sum_{t>i}x_{t} }{4s+1}\right)-63s^3}\\
                &=\frac{1}{96s^3-63s^2+(4s+1)\displaystyle\sum_{t>i}x_{t}}\\   
                &\leq  \frac{1}{33s^3 +(4s+1)\displaystyle\sum_{t>i}x_{t}}  \\
                &\leq \min_{i<j\leq \gamma_P(G)}\frac{j-i}{m_j-m_i}.
                \end{align*}
            Thus, by Lemma~\ref{lemma characterization of beta good sizes}, the size $i$ is indeed useful, as desired.
            
            Now fix some $i\in [s]\setminus R'$, and let $\underline{j},\overline{j}\in R$ be the largest and smallest values respectively such that $\underline{j}<i<\overline{j}$.
            Then using~\eqref{equation upper bound on m_i} and~\eqref{equation lower bound on m_i}, we calculate
                \[ \frac{i-\underline{j}}{m_i-m_{\underline{j}}}\geq \frac{1}{63s^3+ (4s+1)\displaystyle\sum_{t\leq i} x_t-(4s+1)\sum_{t\leq \underline{j}}x_t}=\frac{1}{63s^3}, \]
            and
                \[ \frac{\overline{j}-i}{m_{\overline{j}}-m_i}\leq \frac{s}{(4s+1)\displaystyle\sum_{t\leq \overline{j}}x_t-\left(63s^3+ (4s+1)\displaystyle\sum_{t\leq i} x_t\right)}=\frac{s}{(4s+1)x_{\overline{j}}-63s^3}. \]
            Then, using~\eqref{equation bound on x_i 2}, we have 
            \begin{align*}
                \min_{i<j\leq \gamma_P(G)}\frac{j-i}{m_j-m_i} &\leq \frac{s}{(4s+1)x_{\overline{j}}-63s^3} \\
            & < \frac{s}{(4s+1)\left(\displaystyle\frac{63s^4+63s^3}{4s+1}\right)-63s^3} \\
            & =\frac{1}{63s^3} \\
            & \leq\max_{0\leq j<i}\frac{i-j}{m_i-m_j}
            \end{align*}
            Thus, by Lemma~\ref{lemma characterization of beta good sizes}, $i$ is not a useful size.
            
        \end{proof}
 
\section{Determining how many PMUs are needed}\label{section marginal cost}
    In this section we will develop the ideas of marginal cost and marginal observance and prove Theorem~\ref{theorem when to use gamma sensors}.
    
    \subsection{Marginal cost and marginal observance}
        Given a current placement of sensors which do not suffice to observe the entire graph, it is natural to ask about the cost-benefit analysis for purchasing one more sensor.
        
        Given a graph $G$ and an integer $k\in [\gamma_P(G)]$, we define the \emph{marginal observance}, 
            \[ \mathrm{MObs}(G;k):=\mathrm{maxObs}(G;k)-\mathrm{maxObs}(G;k-1) \]
        and given a cost-ratio $\beta\in\mathbb{R}_{\geq 0}$ the \emph{marginal cost}
            \[ \textrm{MC}(G;k,\beta):=1-\beta\cdot\textrm{MObs}(G;k). \]
        
        It is worth noting that if we fix $G$, $\mathrm{MObs}(G;k)$ is strictly increasing in $k$.
        Furthermore, there is a telescopic relationship between the values of $\mathrm{MObs}$ and $\mathrm{maxObs}$.
        Indeed, for any $k\in [\gamma_P(G)]$, we have
        \begin{equation}\label{equation max obs from marginal obs}
            \mathrm{maxObs}(G;k)=\sum_{i=1}^k \mathrm{MObs}(G;i).
        \end{equation}
        
        Similarly, for the cost function we have
            \[ \mathrm{C}(G;k,\beta)=\sum_{i=1}^k \mathrm{MC}(G;i,\beta). \]
        
        Marginal observance can tell us when certain sizes are useful or not useful.
        The following observation characterizes when $i-1$ sensors are better than $i$ sensors, based on marginal observance and $\beta$.
        
        \begin{proposition}\label{observation marginal cost beta bound on when i-1 is better than i}
            Let $G$ be a graph and $i\in [\gamma_P(G)]$.
            Then $\mathrm{C}(G;i-1,\beta)\leq\mathrm{C}(G;i,\beta)$ if and only if $\beta\leq \frac{1}{\mathrm{MObs}(G;i)}$.
            Furthermore, if one inequality is strict, then so is the other. 
        \end{proposition}
        \begin{proof}
            We can directly calculate that
                \[ \mathrm{C}(G;i,\beta)=\mathrm{C}(G;i-1,\beta)+\mathrm{MC}(G;i)=\mathrm{C}(G;i-1,\beta)+1-\beta\cdot\mathrm{MObs}(G;i). \]
            Thus, $\mathrm{C}(G;i-1,\beta)\leq \mathrm{C}(G;i,\beta)$ if and only if $\beta\leq \frac{1}{\mathrm{MObs}(G;i)}$, and if one inequality is strict, so is the other.
        \end{proof}
        
        Note, Proposition~\ref{observation marginal cost beta bound on when i-1 is better than i} also implies $\mathrm{C}(G;i-1,\beta)\geq\mathrm{C}(G;i,\beta)$ if and only if $\beta\geq\frac{1}{\mathrm{MObs}(G;i)}$.
        We can also utilize this result to rule out a size $i$ from being useful.
        \begin{proposition}\label{proposition marginal obs implies not useful}
            Let $G$ be a graph and $i\in [\gamma_P(G)-1]$.
            If $\mathrm{MObs}(G;i)\leq \mathrm{MObs}(G;i+1)$, then $i$ is not a useful size for $G$.
        \end{proposition}
        \begin{proof}
        	We will proceed via the contrapositive. If $i$ is useful, then there is a range of $\beta$ such that $\mathrm{C}(G;i-1,\beta)\geq\mathrm{C}(G;i,\beta)$ and $\mathrm{C}(G;i+1,\beta)\geq\mathrm{C}(G;i,\beta)$. Using Proposition~\ref{observation marginal cost beta bound on when i-1 is better than i}, this gives us that
        	\[
        	\frac{1}{\mathrm{MObs}(G;i)}\leq \beta\leq \frac{1}{\mathrm{MObs}(G;i+1)}.
        	\]
        	If we had $\frac{1}{\mathrm{MObs}(G;i)}=\frac{1}{\mathrm{MObs}(G;i+1)}$, then we would have that a set of size $i$ is $\beta$-best only at a single value of $\beta$, so $i$ is not useful, and thus we must have $\frac{1}{\mathrm{MObs}(G;i)}< \frac{1}{\mathrm{MObs}(G;i+1)}$, implying that $\mathrm{MObs}(G;i)>\mathrm{MObs}(G;i+1)$.
        \end{proof}
        
        The converse of Proposition~\ref{proposition marginal obs implies not useful} is not true in general; if $\mathrm{MObs}(G;i)>\mathrm{MObs}(G;i+1)$, we can say that there is an interval for $\beta$ for which the cost for $i$ sensors is lower than the cost for $i-1$ or $i+1$ sensors, but it is possible that some other number of sensors has even lower cost on this interval.
        However, if the marginal observances are strictly decreasing, every size is useful.
        \begin{proposition}
            Let $G$ be a graph.
            If $\mathrm{MObs}(G;\ell)>\mathrm{MObs}(G;\ell+1)$ for all $\ell\in [\gamma_P(G)-1]$, then every size in $[\gamma_P(G)]\cup \{0\}$ is useful.
        \end{proposition}
        \begin{proof}
            By Observation~\ref{observation beta small or large}, we have that $0$ and $\gamma_P(G)$ are both useful sizes.
            Now, fix some $i\in [\gamma_P(G)-1]$.
            Towards applying Lemma~\ref{lemma characterization of beta good sizes}, we calculate
            \begin{equation}\label{equation decreasing marginal obs max of sizes below}
                \max_{0\leq j<i}\frac{i-j}{\mathrm{maxObs}(G;i)-\mathrm{maxObs}(G;j)}=\max_{0\leq j<i}\frac{i-j}{\sum_{\ell=j+1}^i\mathrm{MObs}(G;\ell)}.
            \end{equation}
            For a fixed $j$ with $0\leq j<i$, using the generalized mediant inequality, we can write
                \[ \frac{i-j}{\sum_{\ell=j+1}^i\mathrm{MObs}(G;\ell)}=\frac{\sum_{\ell=j+1}^i1}{\sum_{\ell=j+1}^i\mathrm{MObs}(G;\ell)}\leq \max_{j+1\leq \ell\leq i}\frac{1}{\mathrm{MObs}(G;\ell)}=\frac{1}{\mathrm{MObs}(G;i)}, \]
            where on the last equality we use that $\mathrm{MObs}(G;\ell)$ is decreasing.
            Thus, returning to~\eqref{equation decreasing marginal obs max of sizes below}, we have
                \[ 
                \max_{0\leq j<i}\frac{i-j}{\mathrm{maxObs}(G;i)-\mathrm{maxObs}(G;j)} \leq \frac{1}{\mathrm{MObs}(G;i)}. 
                \]
            By essentially the same argument, we have that
                \[ 
                \min_{i<j\leq \gamma_P(G)}\frac{j-i}{\mathrm{maxObs}(G;j)-\mathrm{maxObs}(G;i)} \geq \frac{1}{\mathrm{MObs}(G;i+1)}. 
                \]
            Since $\frac{1}{\mathrm{MObs}(G;i)}<\frac{1}{\mathrm{MObs}(G;i+1)}$, size $i$ is useful by Lemma~\ref{lemma characterization of beta good sizes}.
        \end{proof}
    
    \subsection{Marginal observance and minimum fort size}
        The results in this section will be useful for the proof of Theorem~\ref{theorem when to use gamma sensors}, but they may also be of independent interest.
        The result below is stated for $\underline{f}(G)\geq 2$ below; it is worth noting that this condition is equivalent to $G$ having no isolated vertices.
        \begin{lemma}\label{lemma marginal observation with f=2}
            Let $G$ be a graph with $\underline{f}(G)\geq 2$. For all $i\in [\gamma_P(G)]$, 
                $ \mathrm{MObs}(G;i)\geq 2. $
        \end{lemma}
        \begin{proof}
            Let $S\in \displaystyle\argmax_{S'\in\binom{V(G)}{i-1}} |\mathrm{Obs}(G;S')|$.
            We will demonstrate a vertex $x$ such that $|\mathrm{Obs}(G;S\cup \{x\})|$ is at least two larger than $|\mathrm{Obs}(G;S)|$.
            
            If there exists a connected component $C$ of $G$ with $V(C)\cap \mathrm{Obs}(G;S)=\emptyset$, then for any $x\in V(C)$, we have that $x$ observes itself and at least one neighbor, as there are no isolated vertices.
            
            If on the other hand, no such connected component exists, then there must be an unobserved fort in $V(G)\setminus \mathrm{Obs}(G;S)$ with a non-empty entrance.
            Let $x$ be in the entrance of this fort.
            Note that $x$ has at least two neighbors in the fort that were unobserved by $S$, but are observed by $S\cup \{x\}$.

            In both cases, we have determined some vertex $x$ such that $|\mathrm{Obs}(G;S\cup \{x\})|$ is at least two larger than $|\mathrm{Obs}(G;S)|$.
        \end{proof}
        
        We can refine Lemma~\ref{lemma marginal observation with f=2} if we know there are no forts of size $2$ in $G$.
        Before we do so, we need a characterization of forts of size $2$.
        Two vertices $x,y\in V(G)$ are \emph{twins} if either $N(x)=N(y)$ or $N[x]=N[y]$.
        The following synthesizes folklore results on characterizing small forts.        \begin{lemma}\label{lemma fort size 3 characterization}
            Let $G$ be a graph.
            $\underline{f}(G)\geq 3$ if and only if both $G$ does not contain any isolated vertices and $G$ does not contain any twins.
        \end{lemma}
        \begin{proof}
            For the forward direction, an isolated vertex is a fort of size $1$ and twins constitute a fort of size $2$.
        
            For the reverse direction, assume $G$ is a graph with no isolated vertices and no twins.
            The only possible fort of size $1$ are isolated vertices.
            Similarly if $F$ was a fort of size $2$ in $G$, then every vertex $x\not\in F$ would necessarily be adjacent to either both vertices in $F$ or neither vertex in $F$, so the vertices in $F$ would be twins.
            Thus $\underline{f}(G)\geq 3$.
        \end{proof}
        \begin{lemma}\label{lemma marginal obs with f geq 3}
            Let $G$ be a graph with $\underline{f}(G)\geq 3$. Then $ \mathrm{MObs}(G;i)\geq 3$ for all $i\in [\gamma_P(G)]$.
        \end{lemma}

        \begin{proof}
            Let $S\in \displaystyle\argmax_{S'\in\binom{V(G)}{i-1}} |\mathrm{Obs}(G;S')|$.
            We will demonstrate a vertex $w$ such that $\mathrm{Obs}(G;S\cup \{w\})$ is at least three larger than $|\mathrm{Obs}(G;S)|$.
            Let $X:=V(G)\setminus \mathrm{Obs}(G;S)$.
            
            If any vertex $v$ in $G$ satisfies $N[v]\cap X\geq 3$, then we can choose $w=v$. 

            Now, assume that for all $v\in V(G)$ that $|N[v]\cap X| < 3$.
            Note that this means that $X$ intersects nontrivially with every connected component of $G$: the condition $\underline{f}(G)\geq 3$ implies that $G[X]$ can neither contain $K_1$ nor $K_2$ and the $|N[v]\cap X| <3$ condition prohibits any connected component of $G$ entirely contained in $X$ with 3 or more vertices.
            
            Thus, there exists a vertex $x_1\in X$ with a neighbor $z\in \mathrm{Obs}(G;S)$.
            As $x_1$ was not observed, $z$ must have a second neighbor in $X$, say $x_2$. As $|N[z]\cap X| < 3$, $z$ has no other neighbors in $X$. 
            Moreover, note that $x_1$ and $x_2$ are not twins by Lemma~\ref{lemma fort size 3 characterization}.
            Without loss of generality, we may assume there is some $y\in N(x_1)\setminus N(x_2)$. 

            If $y\in X$, then $N[x_1]\cap X=\{x_1,y\}$ since $|N[x_1]\cap X| < 3$.
            Choose $w=z$.
            In the domination step, $x_1$ and $x_2$ are observed and then $x_1$ observes $y$ by a zero forcing step.

            If $y\not\in X$, then as $x_1$ is unobserved, $y$ must have a second neighbor, $x_3 \in X$.
            Note $N[y]\cap X = \{x_1,x_3\}$.
            Choose $w=y$.
            Then $y$ observes $x_1,x_3$ in the domination step and then $z$ observes $x_2$ by a zero forcing step.
            
            In each case, we have a choice of vertex $w$ such that
                \[ 
                |\mathrm{Obs}(G;S\cup \{w\})|\geq |\mathrm{Obs}(G;S)|+3. \qedhere
                \]
        \end{proof}
        
        Based on the statements of Lemma~\ref{lemma marginal observation with f=2} and Lemma~\ref{lemma marginal obs with f geq 3}, one might ask if in general for $k\in\mathbb{N}$, is it is the case that $\underline{f}(G)\geq k$ implies that $\mathrm{MObs}(G;i)\geq k$ for all $i\in [\gamma_P(G)]$?
        We provide a short discussion of this question in the concluding remarks (Section~\ref{section concluding remarks}).
        In the interim, we mention that this observation is true at least for $i=\gamma_P(G)$.
        \begin{observation}\label{observation marginal obs for gamma is minimal fort}
            Let $G$ be a graph.
            We have that
                \[ \mathrm{MObs}(G;\gamma_P(G))\geq \underline{f}(G). \]
        \end{observation}
        \begin{proof}
            For any set $S\subseteq V(G)$ that is not a power dominating set, the set $V(G)\setminus \mathrm{Obs}(G;S)$ is a fort, so $\mathrm{maxObs}(G;\gamma_P(G)-1)\leq n-\underline{f}(G)$.
            Thus,
                \[ \mathrm{MObs}(G;\gamma_P(G))=\mathrm{maxObs}(G;\gamma_P(G))-\mathrm{maxObs}(G;\gamma_P(G)-1)\geq \underline{f}(G). \]
        \end{proof}
    
    \subsection{Using \texorpdfstring{$\gamma_P(G)$}{the power domination number of} sensors}\label{subsection using gamma sensors}
        In this section, we provide the proof of Theorem~\ref{theorem when to use gamma sensors}, along with many results which may be of independent interest.
        
        \begin{proposition}\label{proposition f = 2 beta bound for gamma sensors}
            Let $G$ be a graph on $n$ vertices with $\underline{f}(G)=2$.
            If $\beta\geq 1/2$, then any minimum power dominating set is $\beta$-best.
        \end{proposition}
        \begin{proof}
            Let $1\leq k\leq \gamma_P(G)-1$.
            Using~\eqref{equation max obs from marginal obs}, we can write 
                \[ \mathrm{maxObs}(G;\gamma_P(G)-k)=\sum_{i=1}^{\gamma_P(G)-k}\mathrm{MObs}(G;i)=\mathrm{maxObs}(G;\gamma_P(G))-\sum_{i=0}^{k-1}\mathrm{MObs}(G;\gamma_P(G)-i) \]
            Then, using the fact that $\mathrm{maxObs}(G;\gamma_P(G))=n$ and using Lemma~\ref{lemma marginal observation with f=2}, we find that
                \[ \mathrm{maxObs}(G;\gamma_P(G)-k)\leq n-2k. \]
            Thus, if $S$ is a minimum power dominating set and $S'\subseteq V(G)$ with $|S'|=\gamma_P(G)-k$, we have that
                \[ \mathrm{C}(G;S,\beta)-\mathrm{C}(G;S',\beta)=\gamma_P(G)-(\gamma_P(G)-k+\beta(n-|\mathrm{Obs}(G;S'))|)\leq k-\beta(2k). \]
            This is non-positive as long as $\beta\geq 1/2$.
            Similarly, 
                \[ \mathrm{C}(G;S,\beta)-\mathrm{C}(G;\emptyset,\beta)=\gamma_P(G)-\beta n. \]
            This is non-positive as long as $\beta\geq \frac{\gamma_P(G)}{n}$.
            Since $\underline{f}(G)=2$, $G$ does not contain any isolated vertices.
            In particular, $\gamma_P(G)\leq n/2$, and thus $\frac{\gamma_P(G)}{n}\leq \frac{1}{2}$.
            Thus, as long as $\beta\geq 1/2$, $S$ is $\beta$-best.
        \end{proof}
        \begin{proposition}\label{proposition f geq 3 beta bound for gamma sensors}
            Let $G$ be a graph on $n$ vertices with $\underline{f}(G)\geq 3$.
            If
                \[ \beta\geq\max\left\{\frac{1}{3}-\frac{\underline{f}(G)-3}{3(\underline{f}(G)+3\gamma_P(G)-6)},\frac{\gamma_P(G)}{n}\right\} \]
            then any minimum power dominating set is $\beta$-best.
        \end{proposition}
        \begin{proof}
            Let $S$ be a minimum power dominating set and set $f:=\underline{f}(G)$.
            If $1\leq k<\gamma_P(G)$, we have from Lemma~\ref{lemma marginal obs with f geq 3} and Observation~\ref{observation marginal obs for gamma is minimal fort}, along with~\eqref{equation max obs from marginal obs}, we can write
            \begin{align*}
                \mathrm{maxObs}(G;\gamma_P(G)-k)=\sum_{i=1}^{\gamma_P(G)-k}\mathrm{MObs}(G;i)&=\mathrm{maxObs}(G;\gamma_P(G))-\sum_{i=0}^{k-1}\mathrm{MObs}(G;\gamma_P(G)-i)\\
                &\leq n-3(k-1)-f,
            \end{align*}
            so if $S'\subseteq V(G)$ is any set with $|S'|=\gamma_P(G)-k$, then
            \begin{align*}
                \mathrm{C}(G;S,\beta)-\mathrm{C}(G;S',\beta)=\gamma_P(G)-\big(\gamma_P(G)-k+\beta(n-|\mathrm{Obs}(G;S')|\big)&\leq k-\beta(3(k-1)+f),
            \end{align*}
            which is non-positive as long as 
                \[ \frac{1}{3}-\frac{f-3}{3(f+3k-3)}=\frac{k}{3(k-1)+f}\leq \beta \]
            If $k=\gamma_P(G)$, then we have $S'=\emptyset$, so
            \begin{align*}
                \mathrm{C}(G;S,\beta)-\mathrm{C}(G;S',\beta)&=\gamma_P(G)-\beta n,
            \end{align*}
            which is non-positive as long as $\beta\geq \frac{\gamma_P(G)}{n}$. Thus, $S$ is $\beta$-best for any $\beta$ satisfying
                \[ \beta\geq \max\left\{\max_{1\leq k<\gamma_P(G)}\frac{1}{3}-\frac{f-3}{3(f+3k-3)},\frac{\gamma_P(G)}{n}\right\}=\max\left\{\frac{1}{3}-\frac{f-3}{3(f+3\gamma_P(G)-6)},\frac{\gamma_P(G)}{n}\right\}. \]
        \end{proof}

        We can now prove Theorem~\ref{theorem when to use gamma sensors}.
        \begin{theorem}\label{theorem when to use gamma sensors}
            Let $G$ be a graph on $n$ vertices. 
            Let 
            \[
            B:=\begin{cases}
                1&\text{ if }\underline{f}(G)=1,\\
                \frac{1}{2}&\text{ if }\underline{f}(G)=2,\\
                \frac{1}{3}-\frac{\underline{f}(G)-3}{3(\underline{f}(G)+3\gamma_P(G)-6)}&\text{ if }\underline{f}(G)\geq 3.
            \end{cases}
            \]
            If $\displaystyle \beta\geq \max\left\{B,\frac{\gamma_P(G)}{n}\right\}$, then any minimum power dominating set is $\beta$-best.
            Furthermore, this is best-possible if $\underline{f}(G)\leq 3$.
        \end{theorem}

        \begin{proof}
            Observation~\ref{observation beta small or large} implies the result for $\underline{f}(G) = 1$ as $\frac{\gamma_P(G)}{n} \leq 1$.                      
            Proposition~\ref{proposition f = 2 beta bound for gamma sensors} and Proposition~\ref{proposition f geq 3 beta bound for gamma sensors} give the result for $\underline{f}(G)=2$ and $\underline{f}(G)\geq 3$ respectively.
            We will now show these bounds are best possible. 
            
            \textbf{Case 1:} $\underline{f}(G)=1$.
            The bound is sharp for every such graph.
            Indeed, if $\underline{f}(G)=1$, then $G$ contains an isolated vertex, and it is straightforward to compute that $\mathrm{C}(G;\gamma_P(G)-1,\beta)=\gamma_P(G)-1+\beta<\gamma_P(G)=\mathrm{C}(G;\gamma_P(G),\beta)$ when $\beta<1$.
            
            \textbf{Case 2:} $\underline{f}(G)=2$.
            Let $n$ be even and consider $\frac{n}{2}K_2$.
            We can compute $\mathrm{C}(\frac{n}{2}K_2;k,\beta)=k+\beta(n-2k)=\beta n+k(1-2\beta)$, which for $\beta<1/2$ is minimized at $k=0$ (rather than at $k=\gamma_P(G)$).
            
            \textbf{Case 3:} $\underline{f}(G)\geq 3$.
            Let $6\mid n$, and consider the graph $G$ obtained from $\frac{n}{6}C_6$ by choosing three independent vertices from each copy of $C_6$ to form a set $K$ of $n/2$ vertices, and then adding edges to make $K$ a clique.
            Let $X_1,X_2,\dots,X_{n/6}$ denote the sets of size $3$ such that $X_i$ is disjoint from $K$, and the three vertices in each $X_i$ were all in the same initial $C_6$.
            It is straightforward to see that $G$ does not contain any isolated vertices or twins, so by Lemma~\ref{lemma fort size 3 characterization}, $\underline{f}(G)\geq 3$, and indeed, $\underline{f}(G)=3$ as $X_i$ is a fort of size $3$ for each $i$.
            
            Since each $X_i$ is a fort, and the sets $\mathfrak{t}(X_i)$ are pairwise-disjoint, we have that $\gamma_P(G)\geq n/6$, and indeed, it is straightforward to see that if we choose $n/6$ vertices from $K$, one in each entrance of a set $X_i$, this is a power dominating set, so $\gamma_P(G)=n/6$.
            We also note that if $v\in K$ is any vertex, then $|\mathrm{Obs}(G,v)|=n/2+3$. With this, if $\beta<1/3$, we can compute
                \[ \mathrm{C}(G;\{v\},\beta)=1+\beta(n/2-3)<1+n/6-1=\gamma_P(G)=\mathrm{C}(G;\gamma_P(G),\beta). \qedhere \]
        \end{proof}

        \begin{figure}[htbp]
            \centering
            \resizebox{!}{.5\textwidth}{
                \begin{tikzpicture}
                    \node [style=Red Node] (0) at (-1.5, 2) {$6$};
                    \node [style=Black Node] (1) at (-1.5, 4) {$7$};
                    \node [style=Red Node] (2) at (0, 2) {$8$};
                    \node [style=Black Node] (3) at (0, 4) {$11$};
                    \node [style=Red Node] (4) at (1.5, 2) {$10$};
                    \node [style=Black Node] (5) at (1.5, 4) {$9$};
                    \node [style=Red Node] (6) at (-2.75, -0.5) {$0$};
                    \node [style=Black Node] (7) at (-4.5, -2) {$1$};
                    \node [style=Red Node] (8) at (-2, -1.5) {$2$};
                    \node [style=Black Node] (9) at (-3.5, -3) {$5$};
                    \node [style=Red Node] (10) at (-1.25, -2.5) {$4$};
                    \node [style=Black Node] (11) at (-2.5, -4) {$3$};
                    \node [style=Red Node] (12) at (2.75, -0.5) {$12$};
                    \node [style=Black Node] (13) at (4.5, -2) {$13$};
                    \node [style=Red Node] (14) at (2, -1.5) {$14$};
                    \node [style=Black Node] (15) at (3.5, -3) {$17$};
                    \node [style=Red Node] (16) at (1.25, -2.5) {$16$};
                    \node [style=Black Node] (17) at (2.5, -4) {$15$};
                    \draw [style=Black Alpha Edge] (6) to (7);
                    \draw [style=Black Alpha Edge] (7) to (8);
                    \draw [style=Black Alpha Edge] (8) to (11);
                    \draw [style=Black Alpha Edge] (11) to (10);
                    \draw [style=Black Alpha Edge] (10) to (9);
                    \draw [style=Black Alpha Edge] (9) to (6);
                    \draw [style=Black Alpha Edge] (0) to (1);
                    \draw [style=Black Alpha Edge] (1) to (2);
                    \draw [style=Black Alpha Edge] (2) to (5);
                    \draw [style=Black Alpha Edge] (5) to (4);
                    \draw [style=Black Alpha Edge] (4) to (3);
                    \draw [style=Black Alpha Edge] (3) to (0);
                    \draw [style=Black Alpha Edge] (16) to (17);
                    \draw [style=Black Alpha Edge] (17) to (14);
                    \draw [style=Black Alpha Edge] (14) to (13);
                    \draw [style=Black Alpha Edge] (13) to (12);
                    \draw [style=Black Alpha Edge] (12) to (15);
                    \draw [style=Black Alpha Edge] (15) to (16);
                    \draw [style=Black Alpha Edge] (0) to (6);
                    \draw [style=Black Alpha Edge] (6) to (2);
                    \draw [style=Black Alpha Edge] (2) to (8);
                    \draw [style=Black Alpha Edge] (8) to (4);
                    \draw [style=Black Alpha Edge] (4) to (10);
                    \draw [style=Black Alpha Edge] (10) to (0);
                    \draw [style=Black Alpha Edge] (0) to (8);
                    \draw [style=Black Alpha Edge] (8) to (16);
                    \draw [style=Black Alpha Edge] (16) to (0);
                    \draw [style=Black Alpha Edge] (0) to (14);
                    \draw [style=Black Alpha Edge] (14) to (2);
                    \draw [style=Black Alpha Edge] (2) to (12);
                    \draw [style=Black Alpha Edge] (12) to (4);
                    \draw [style=Black Alpha Edge] (4) to (14);
                    \draw [style=Black Alpha Edge] (14) to (16);
                    \draw [style=Black Alpha Edge] (16) to (4);
                    \draw [style=Black Alpha Edge] (4) to (6);
                    \draw [style=Black Alpha Edge] (6) to (16);
                    \draw [style=Black Alpha Edge] (16) to (10);
                    \draw [style=Black Alpha Edge] (10) to (14);
                    \draw [style=Black Alpha Edge] (14) to (8);
                    \draw [style=Black Alpha Edge] (8) to (12);
                    \draw [style=Black Alpha Edge] (12) to (6);
                    \draw [style=Black Alpha Edge] (6) to (8);
                    \draw [style=Black Alpha Edge] (8) to (10);
                    \draw [style=Black Alpha Edge] (10) to (12);
                    \draw [style=Black Alpha Edge] (12) to (14);
                    \draw [style=Black Alpha Edge] (10) to (2);
                    \draw [style=Black Alpha Edge] (16) to (2);
                    \draw [style=Black Alpha Edge] (14) to (6);
                    \draw [style=Black Alpha Edge] (12) to (0);
                    \draw [style=Black Alpha Edge] (0) to (2);
                    \draw [style=Black Alpha Edge] (2) to (4);
                    \draw [style=Black Alpha Edge, bend left=333] (10) to (6);
                    \draw [style=Black Alpha Edge, bend left=333] (12) to (16);
                    \draw [style=Black Alpha Edge, bend right=335] (4) to (0);
                    
                    \node [style=Text Node] (17) at (-5.375, -1.25) {$X_1$};
                    \begin{scope}[rotate around={-45:(-3.5, -3)}]
                        \draw [rounded corners=10pt, style=Dashed Edge] (-6.5,-3.5) rectangle (-1.5, -2.5);
                    \end{scope}
                    
                    \node [style=Text Node] (17) at (-2.625, 4) {$X_2$};
                    \draw [rounded corners=10pt, style=Dashed Edge] (-3, 3.5) rectangle (2, 4.5);
                    
                    \node [style=Text Node] (17) at (5,-1.25) {$X_3$};
                    \begin{scope}[rotate around={45:(3.5, -3)}]
                        \draw [rounded corners=10pt, style=Dashed Edge] (1.5, -3.5) rectangle (6.5, -2.5);
                    \end{scope}
                \end{tikzpicture}
            }
            \caption{The graph $G$ as described in \textbf{Case 3} with $K = \{0,2,4,6 , 8, 10, 12, 14, 16\}$ indicated in red, and $X_1 = \{1,3,5\}$, $X_2 = \{7,9,11\}$, and $X_3 = \{13,15,17\}$ identified by the dashed curves.}
        \end{figure}
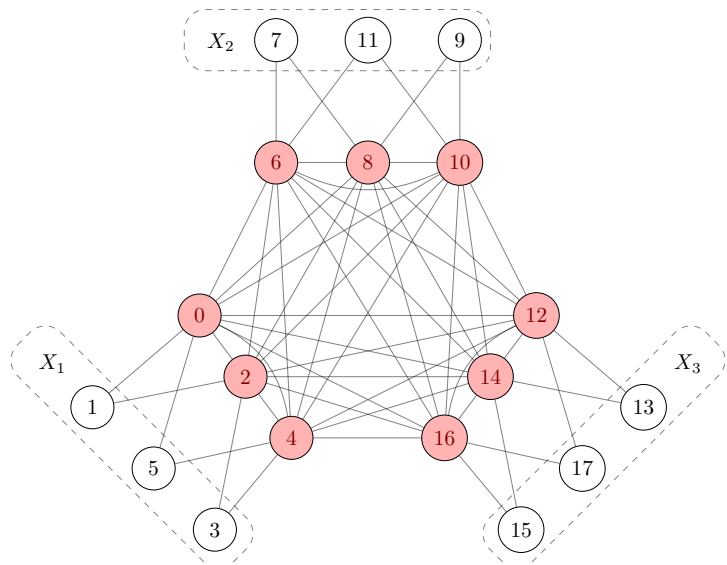

        We conclude this section with a result which may be helpful for finding a minimum power dominating set.
        Power domination can be phrased as a cover problem; given a graph $G$, a set $S$ is a power dominating set if and only if $S$ intersects $\mathfrak{t}(F)$ for all forts $F\subseteq V(G)$~\cite{SH2020}.
        Of course, one can strengthen this to include only minimal forts in $G$.
        A common heuristic is that that placing a sensor in the entrance of a fort, rather than in the fort itself, is often better.
        The following theorem gives some evidence for this heuristic; we can always find a power dominating set which is contained inside the union of the entrances of the minimal forts of $G$.
        Or, in the converse, we can exclude any vertices which are not in the entrance of a minimal fort from our search, and we are guaranteed to still be able to find a minimum power dominating set among the non-excluded vertices.

        \begin{theorem}
            For any connected graph $G$ with at least three vertices, there exists a minimum power-dominating set $S$ along with a collection of minimal forts $F_1,F_2,\dots, F_k$ such that $S$ is contained in the union of the entrances of the $F_i$'s.
        \end{theorem}
        \begin{proof}
            Let $\mathfrak{F}$ denote the collection of all minimal forts of $G$.
            Given a vertex $x\in V(G)$, let $\mathfrak{C}(x)\subseteq\mathfrak{F}$ denote the set of minimal forts $F$ such that $x\in \mathfrak{t}(F)$.
            Finally, given a set $S\subseteq V(G)$ and a vertex $x\in V(G)\setminus S$, let $\mathfrak{R}(x,S)\subseteq \mathfrak{F}$ be the set of minimal forts $F$ such that $x\in\mathfrak{t}(F)$, but $S\cap\mathfrak{t}(F)=\emptyset$.
            Given a vertex $x$ and a set $T\subseteq V(G)$, let $d(x,T)$ denote the distance from $x$ to $T$.
        
            Assume to the contrary that the statement is not true, and let $G$ be a graph without this property.
            Out of all minimum power-dominating sets of $G$, let $S$ be one with the following properties:
            \begin{enumerate}
                \item\label{property1}
                    The number of vertices in $S$ which are not in the entrance of any minimal fort is minimized.
                \item\label{property2}
                    Subject to Property~\ref{property1}, $S$ contains a vertex $w$ which is not in the entrance of any minimal fort, but the distance from $w$ to $\displaystyle A(w,S):=\bigcup_{F\in \mathfrak{R}(w,S\setminus \{w\})}\mathfrak{e}(F)$ is minimized. 
            \end{enumerate}
        
            Let $y\in N[w]$ be a vertex such that $d(y,A(w,S))=d(w,A(w,S))-1$.
            We claim that $\mathfrak{R}(w,S\setminus \{w\})\subseteq \mathfrak{C}(y)$.
            Indeed, for any minimal fort $F\in \mathfrak{R}(w,S\setminus \{w\})$, we must have $w\in F$ since we do not have $w\in \mathfrak{e}(F)$, and so $y$ is adjacent to a vertex in $F$, giving us that $y\in \mathfrak{t}(F)$.
        
            With this in hand, we claim that $S'=(S\setminus \{w\})\cup \{y\}$ is a power dominating set and that $A(y,S')=A(w,S)$.
            Indeed, any fort $F\in \mathfrak{F}\setminus\mathfrak{R}(w,S\setminus \{w\})$ has a vertex from $S\setminus \{w\}$ in $\mathfrak{t}(F)$, while all the forts $F\in\mathfrak{R}(w,S\setminus \{w\})$ have $y$ in $\mathfrak{t}(F)$.
            Thus, $S'$ is a power dominating set, and further is minimal since $|S'|=|S|=\gamma_P(G)$.
            The fact that $A(y,S')=A(w,S)$ follows from the fact that 
                \[ \mathfrak{R}(y,S'\setminus \{y\})=\mathfrak{F}\setminus \left(\bigcup_{a\in S'\setminus \{y\}}\mathfrak{C}(a)\right)=\mathfrak{F}\setminus \left(\bigcup_{a\in S\setminus \{w\}}\mathfrak{C}(a)\right)=\mathfrak{R}(w,S\setminus \{w\}). \]
            This implies that $d(y,A(y,S'))=d(y,A(w,S))=d(w,A(w,S))-1$. 
            If $y$ is in the entrance of any minimal fort, we have a contradiction of Property~\ref{property1}. 
            Otherwise, the number of vertices in $S'$ which are not in the entrance of any minimal fort is the same as in $S$, but also $d(y,A(u,S')) < d(w,A(w,S))$. 
            Therefore, $S'$ and $y$ contradict the minimality of the choice of $S$ and $w$ in Property~\ref{property2}.
        \end{proof}

\section{Concluding remarks}\label{section concluding remarks}
    There are still many open theoretical questions involving the cost function and maximum observance function.
    
    The first involves an extension of Lemmas~\ref{lemma marginal observation with f=2} and~\ref{lemma marginal obs with f geq 3}.
    Both statements show that if we have a bound on the minimum fort size of a graph, we also have a bound on the marginal observances of that graph.
    It is possible that a similar phenomenon happens with larger minimum fort sizes. We ask the following question.
    \begin{conjecture}
        Let $G$ be a graph. If $\underline{f}(G)\geq 4$, then
    	   $ \mathrm{MObs}(G;i)\geq 4 $
        for all $i\in [\gamma_P(G)]$.
    \end{conjecture}
    
    If this is true, one might ask if the seemingly natural bound $\mathrm{MObs}(G;i)\geq \underline{f}(G)$ for all $i\in [\gamma_P(G)]$ holds.
    However, this is not true for all graphs.
    Consider the square grid graph $P_n\square P_n$ with $n>6$ which has $\underline{f}(P_n\square P_n)=n$ by \cite[Theorem 9]{LLC2025}.
    A single sensor can only observe at most $6$ vertices, giving us $\mathrm{MObs}(P_n\square P_n,1)=6< n= \underline{f}(P_n\square P_n)$.
    
    The next question we would like to see answered involves computational complexity.
    To determine the entire cost function for a given graph $G$, one first needs to determine the power domination number $G$.
    It is NP-complete to determine the power domination number of a graph~\cite{HHHH2002}, so determining the cost function for a general graph is likely computationally difficult.
    However, determining the power domination number for trees can be done in linear time~\cite{HHHH2002}.
    The known polynomial-time algorithms for finding a minimum power dominating set in a tree $T$ strongly rely on being able to make certain locally optimal decisions that also turn out to be globally optimal.
    For example, if you have a vertex adjacent to two leaves in your graph, you can safely assume that this vertex appears in a minimum power dominating set.
    This property is not necessarily true when calculating a $\beta$-best set, or in general when calculating $\mathrm{maxObs}(T;k)$.
    As such, we ask the following:
    \begin{question}
        Let $T$ be a tree.
        Is it $NP$-complete to determine $\mathrm{maxObs}(T;k)$ for $k < \gamma_P(G)$? 
    \end{question}

\section{Acknowledgements}
    This project was sponsored, in part, by the Air Force Research Laboratory via the Autonomy Technology Research Center, University of Dayton, and Wright State University.
    This research was also supported by Air Force Office of Scientific Research labtask award 23RYCOR004.

\bibliographystyle{plain}
\bibliography{bib}

\end{document}